\documentclass{amsart}

\usepackage{amsmath}
\usepackage{amsthm}
\usepackage{amssymb}
\usepackage{xypic}

\usepackage{tikz}
\usepackage{graphicx}
\usepackage{here}
\usetikzlibrary{decorations,decorations.pathmorphing}
\usetikzlibrary{patterns}

\usepackage{enumerate}

\theoremstyle{plain}
\newtheorem{theorem}{Theorem}[section]
\newtheorem{proposition}[theorem]{Proposition}
\newtheorem{corollary}[theorem]{Corollary}
\newtheorem{lemma}[theorem]{Lemma}

\theoremstyle{definition}
\newtheorem{definition}[theorem]{Definition}
\newtheorem{notation}[theorem]{Notation}

\theoremstyle{remark}
\newtheorem{remark}[theorem]{Remark}
\newtheorem{example}[theorem]{Example}

\newcommand{\bR}{\mathbb{R}}
\newcommand{\bZ}{\mathbb{Z}}
\newcommand{\bT}{\mathbb{T}}
\newcommand{\bB}{\mathbb{B}}

\newcommand{\Area}{{\mathrm{Area}}}
\renewcommand{\div}{{\mathrm{div}}}
\newcommand{\recdiv}{{\mathrm{recdiv}}}

\newcommand{\vol}{{\mathrm{vol}}}
\newcommand{\mComp}{{\mathrm{mComp}}}
\newcommand{\fComp}{{\mathrm{fComp}}}

\newcommand{\Conv}{\mathop{\mathrm{Conv}}\nolimits}
\newcommand{\Newt}{\mathop{\mathrm{Newt}}\nolimits}

\newcommand{\Aff}{\mathop{\mathrm{Aff}}}

\newcommand{\cF}{\mathcal{F}}

\newcommand{\Div}{\mathrm{Div}}
\newcommand{\Divp}{\mathrm{Div}^+}
\newcommand{\Divs}{\mathrm{Div}^{\#}}

\newcommand{\bi}{\mathbf{i}}

\title[Minimum volumes of tropical rational functions]
{Minimum volumes of tropical rational functions}

\author{Masayuki Sukenaga}
\address{
Department of Mathematics, Graduate School of Science, Hiroshima University, 
1-3-1 Kagamiyama, Higashi-Hiroshima, 739-8526 JAPAN}
\email{sukenaga@hiroshima-u.ac.jp}



\begin{document}

\begin{abstract}
When a tropical rational function $\varphi$ on $\bR^n$ is given, we can represent it as a tropical quotient $\varphi=f\oslash g$ with tropical polynomials $f$ and $g$.
We develop the duality theorem for tropical rational functions to define a volume of the pair $(f, g)$.
The first part concerns the case $n=1$.
We show that when $n=1$, we can find a ``unique'' representation of $\varphi(x) \neq -\infty$ as $f(x)\oslash g(x)$ with the pair $(f, g)$ of minimum volume.
However, when $n=2$, this is not true.
\end{abstract}

\maketitle


\section{Introduction}
Tropical algebra is the set $\bT:=\bR\cup\{-\infty\}$ equipped with two operations, $\oplus(=\max)$ and $\odot(=+)$.
We can consider polynomials with coefficients in $\bT$, known as tropical polynomials.
A tropical polynomial defines a map through substitution and a tropical rational function is a function expressed in the form $f\oslash g\ (=f-g)$ by tropical polynomials $f$ and $g\neq -\infty$.
A tropical rational function is also called a piecewise integer affine function (see \cite[Corollary 4.4.12]{MR}), and this has been studied in fields such as machine learning and statistics.
In \cite{TW}, they give two notions of complexity, monomial complexity for a tropical polynomial and factorization complexity for a finite sequence of tropical polynomials.
Furthermore, they define a preorder $\leq_{\mComp}$ on the set of pairs of tropical polynomials as $(f_1, g_1)\leq_{\mComp}(f_2, g_2)$ if and only if $\mComp(f_1)\leq \mComp(f_2)$ and $\mComp(g_1)\leq \mComp(g_2)$, where $\mComp$ is a function from the set of tropical polynomials to $\bZ_{\geq 0}$ that measures a ``monomial complexity'' of a tropical polynomial.
They considered how to find a pair that is small under these orders.
The difficulty is that there are many pairs that are not comparable.
In this paper, we overcome this difficulty by considering a ``volume'' of a pair of tropical polynomials.
For a tropical polynomial $f=\bigoplus_{\mathbf{i}\in \mathbb{Z}^n} c_{\mathbf{i}} \mathbf{x}^{\mathbf{i}}$, the Newton polytope $\Newt(f)$ is the convexhull of the finite set $\{\bi \in \bZ^n\ |\ c_{\bi}\neq-\infty\}$.
It is well known that for a tropical polynomial $f$, there is a duality between the set of ``tropical zero points'' $V(f)$ and the subdivision of $\Newt(f)$ obtained by the coefficients of $f$.
We show the same ``duality theorem'' of tropical rational functions.

\begin{theorem}(=Theorem \ref{f/g})
Let $\varphi$ be a tropical rational function and $f\in\bT[x_1^{\pm1}, \dots, x_n^{\pm1}]$ and $g\in\bT[x_1^{\pm1}, \dots, x_n^{\pm1}]\setminus \{-\infty\}$ tropical polynomials satisfying $\varphi=f\oslash g$.
Then, we have the following:
\begin{eqnarray*}
V(f\oplus (x_{n+1}\odot g))&=&\{(\mathbf{x}, \varphi(\mathbf{x}))\in \bR^{n+1}\ |\ \mathbf{x} \in \bR^n, \varphi(\mathbf{x})\neq -\infty\}\\
&&\cup \{(\mathbf{x}, x_{n+1})\in \bR^{n+1}\ |\ \mathbf{x}\in V(f), x_{n+1}< \varphi(\mathbf{x})\}\\
&&\cup \{(\mathbf{x}, x_{n+1})\in \bR^{n+1}\ |\ \mathbf{x}\in V(g), x_{n+1}> \varphi(\mathbf{x})\}.
\end{eqnarray*}
\end{theorem}

In ordinary algebraic geometry, the graph of a rational function $f/g$ is the set of zero points of $f-x_{n+1}g$.
By the above theorem, the set of tropical zero points of $f\oplus (x_{n+1}\odot g)$ contains the information about the graph of $f\oslash g$ and the sets of tropical zero points of $f$ and $g$.
This is one of the reasons why the volume of a pair of tropical polynomials $(f, g)$ is defined as the volume of $\Newt(f\oplus (x_{n+1}\odot g))$.
The following theorem provides another reason for this definition.

\begin{theorem}(=Theorem \ref{koushitenjogen})
Let $f(x, y), g(x, y)\neq -\infty$ be two variable tropical polynomials.
Assume that if either $\Newt(f)\subset \bR^2$ or $\Newt(g)\subset \bR^2$ has area zero, then it contains at most two lattice points.
Let $A_M$ denote the one of $\Newt(f)$ and $\Newt(g)$ with the larger area, and $A_m$ the one with the smaller area.
Then, $A_M$ contains at most $6\vol(\Newt(f\oplus (z\odot g)))+2$ lattice points, and $A_m$ contains at most $3\vol(\Newt(f\oplus (z\odot g)))+2$ lattice points.
\end{theorem}

It is known that for a tropical polynomial $f$, $\mComp(f)$ is at most the number of lattice points in $\Newt(f)$.
When $n=1$, the following claim is clear by Pick's theorem and when $n=2$, this follows from the above theorem.
The smaller the volume of $\Newt(f\oplus (x_{n+1}\odot g))$, the fewer the number of lattice points in $\Newt(f)$ and $\Newt(g)$ tends to be, so the monomial complexities of $f$ and $g$ also tend to be smaller.

It was shown in \cite{TW} that for a one-variable tropical rational function $\varphi(x)$, there is a unique representation $\varphi(x)=f(x)\oslash g(x)$ that minimizes both the monomial complexity and the factorization complexity.
When $n=1$, there is, in a sense, a unique representation $\varphi(x)=f(x)\oslash g(x)$ that minimizes the volume of $(f, g)$.
This is the first main theorem.

\begin{theorem}(=Theorem \ref{thm_main1})
For any tropical rational function $\varphi(x) \neq -\infty$ on $\bR$, there is an expression $\varphi(x)=f(x)\oslash g(x)$ with the minimum volume.
If $f'(x)$ and $g'(x)$ also satisty $\varphi(x)=f'(x)\oslash g'(x)$ and the volume of $(f', g')$ is equal to that of $(f, g)$, then the dual subdivision of $f'\oplus(y\odot g')$ is a translation of the dual subdivision of $f\oplus(y\odot g)$.
In other words, the dual subdivision of $f\oplus(y\odot g)$ is unique up to translation.
\end{theorem}

On the other hand, when $n=2$, there is a case where the minimum volume expression $\varphi=f\oslash g$ is not uniquely determined.
This is the second main result.

\begin{proposition}(=Proposition \ref{thm_main2})
There is a tropical rational function on $\bR^2$, which has at least two different positive minimum volume expressions.
\end{proposition}

On minimal factorization complexities, the following open question (\cite[Open Question 2]{TW}) is presented: ``Do tropical polynomials have unique irreducible factorization with the minimal factorization complexity?''
We show that this is not true.

\begin{proposition}(=Proposition \ref{thm_main3})
The tropical polynomial $0x^2y^3\oplus 0xy^4\oplus 0x^2y^2\oplus 0xy^3\oplus 0x^2y\oplus 0xy^2\oplus 0y^3\oplus 0xy\oplus 0y^2\oplus 0x\oplus 0y$ has at least two different irreducible factorizations with the minimal factorization complexity.
\end{proposition}

The rest of this paper is organized as follows.
In Section 2, we present basic definitions and facts about tropical hypersurfaces and give a counterexample to the above question.
In Section 3, we show the duality theorem for tropical rational functions, define a volume of a pair of tropical polynomials and prove the main theorems.

\section*{Acknowledgements}
The author is grateful to Nobuyoshi Takahashi for helpful comments.

\section{Tropical hypersurfaces}

In this section, we recall the basics about tropical hypersurfaces.
For details, see \cite{Jos} and \cite{MS}.
First, we look at the definition of tropical algebra.
\begin{definition}[Tropical algebra]
\textit{Tropical algebra} is the triple $(\mathbb{T}, \oplus, \odot)$, where $\bT:=\bR\cup \{-\infty\}$ and the tropical sum $\oplus$ and the tropical multiplication $\odot$ are defined as follows:
\begin{eqnarray*}
x\oplus y&:=&\max\{x, y\},\\
x\odot y&:=&x+y\ (\text{the usual sum}).
\end{eqnarray*}
Tropical algebra satisfies all field axioms except for the existence of the inverse for addition, and is called a tropical semi-field.
For $x\in \bT$ and $y\in \bT\setminus \{-\infty \}=\bR$, we define $x\oslash y:=x-y$.
\end{definition}

Next, we define tropical polynomials.
\begin{definition}[Tropical polynomials]\label{troppoly}
Let $n\in \bZ_{>0}$.
An $n$-variable \textit{tropical Laurent polynomial} $f$ is an expression of the form
\[
f=\bigoplus_{\mathbf{i}\in \mathbb{Z}^n} c_{\mathbf{i}} \mathbf{x}^{\mathbf{i}},
\]
where $c_{\mathbf{i}}\in \mathbb{T}$ and $\mathbf{x}^{\mathbf{i}}$ denotes $x_1^{i_1}\cdots x_n^{i_n}$ for $\mathbf{i}=(i_1, \dots, i_n)$, and only finitely many of the coefficients $c_{\mathbf{i}}$ are not $-\infty$.
We may drop terms with coefficients $-\infty$, and when $c_{\bi}=0$, write $0\mathbf{x}^{\bi}$ as $\mathbf{x}^{\bi}$.
We will refer to a tropical Laurent polynomial simply as a tropical polynomial.
A tropical polynomial defines a map from $\mathbb{R}^n$ to $\mathbb{R}\cup \{-\infty\}$ in the following way:
\[
f(t_1, \dots, t_n)=\max_{\mathbf{i}\in \bZ^n}(c_{\mathbf{i}}+i_1t_1+\dots+i_nt_n).
\]
We write $\bT[x_1^{\pm1}, \dots, x_n^{\pm1}]$ for the set of all $n$-variable tropical polynomials, and define addition and multiplication in a natural way.
\end{definition}

\begin{definition}
We define
\[
\bB[x_1^{\pm1}, \dots, x_n^{\pm1}]:=\left\{\bigoplus_{\bi\in \bZ^n} c_{\bi} \mathbf{x}^{\bi}\in \bT[x_1^{\pm1}, \dots, x_n^{\pm1}]\ \middle|\ c_{\bi}\neq-\infty \Rightarrow c_{\bi}=0\right\}.
\]
Here, the reason for using the symbol $\bB$ is that it is the initial letter of ``Boolean semiring,'' which consists only of $-\infty, 0\in \bT$.
\end{definition}

\begin{definition}[Tropical rational functions]\label{trf}
For $f\in \bT[x_1^{\pm1}, \dots, x_n^{\pm1}]$ and $g\in \bT[x_1^{\pm1}, \dots, x_n^{\pm1}]\setminus \{-\infty \}$, we define the map $f\oslash g$ from $\mathbb{R}^n$ to $\mathbb{R}\cup \{-\infty\}$ in the following way:
\[
(f\oslash g)(t_1, \dots, t_n)=f(t_1, \dots, t_n)-g(t_1, \dots, t_n).
\]
A function $\varphi$ from $\mathbb{R}^n$ to $\mathbb{R}\cup \{-\infty\}$ is called a \textit{tropical rational function} if there exist $f\in \bT[x_1^{\pm1}, \dots, x_n^{\pm1}]$ and $g\in \bT[x_1^{\pm1}, \dots, x_n^{\pm1}]\setminus \{-\infty \}$ such that $\varphi=f\oslash g$.
\end{definition}

\begin{definition}[Tropical hypersurfaces]
Let $f=\bigoplus_{\mathbf{i}} c_{\mathbf{i}} \mathbf{x}^{\bi}\neq -\infty$ be an $n$-variable tropical polynomial.
The tropical polynomial $f$ defines the \textit{tropical hypersurface} $V(f)$ in the following way:
\begin{eqnarray*} 
V(f):=\left\{ (t_1,\dots, t_n)\in \mathbb{R}^n \middle| 
\begin{array}{l}
\text{$\exists \mathbf{i}=(i_1, \dots, i_n),\ \mathbf{j}=(j_1, \dots, j_n) \in \mathbb{Z}^n$ ($\mathbf{i} \neq \mathbf{j}$) s.t.}\\
\text{$c_{\mathbf{i}}+i_1t_1+\dots+i_nt_n=c_{\mathbf{j}}+j_1t_1+\dots+j_nt_n$}\\
\hspace{32.5mm} =f(t_1,\dots, t_n)
\end{array}
\!\right\}.
\end{eqnarray*}
When $f=-\infty$, i.e., all the coefficients of $f$ are $-\infty$, we define $V(-\infty)=\mathbb{R}^n$.
When $n=2$ and $V(f)\neq \emptyset, \bR^2$, the tropical hypersurface $V(f)$ is called a \textit{tropical plane curve}.
\end{definition}

\begin{notation}
For a subset $S\subset \bR^n$, we denote the convex hull of $S$ by $\Conv(S)$ and the affine span of $S$ by $\Aff(S)$.
\end{notation}

\begin{definition}[Dual subdivisions, {\cite[Definition 3.10]{Mik}}]\label{dualsubdiv}
Let $f=\bigoplus_{\mathbf{i}\in \bZ^n} c_{\mathbf{i}} \mathbf{x}^{\bi}$ be a tropical polynomial.
We write $\Newt(f)\subset \mathbb{R}^n$ for $\Conv\left(\{\mathbf{i} \in \mathbb{Z}^n\ |\ c_{\mathbf{i}}\neq -\infty\}\right)$ and $\Delta^{\uparrow}_f\subset \mathbb{R}^{n+1}=\bR^n\times \bR$ for $\Conv\left(\{ (\mathbf{i}, \alpha)\in \mathbb{Z}^n\times \mathbb{R}\ \mid \alpha \leq c_{\bi}\}\right)$.
Then, the projections of the bounded faces of $\Delta^{\uparrow}_f$ to $\bR^n$ form a lattice subdivision of $\Newt(f)$.
We call this the \textit{dual subdivision} of $f$ and denote it by $\Delta_f$.
\end{definition}

\begin{theorem}[Duality Theorem, {\cite[Proposition 3.11]{Mik}}]\label{dualitytheorem}
Let $f=\bigoplus_{\mathbf{i}} c_{\mathbf{i}} \mathbf{x}^{\bi}$ be an $n$-variable tropical polynomial satisfying $V(f)\neq \emptyset, \bR^n$.
Then, $V(f)$ is the support of a finite $(n-1)$-dimensional polyhedral complex $\Sigma_f$ (with possibly noncompact cells) in $\mathbb{R}^n$.
It is dual to $\Delta_f$ in the following sense:
\begin{itemize}
  \item (Closures of) domains of $\bR^n\setminus V(f)$ correspond to lattice points in $\Delta_f$.
  \item For $0\leq d\leq n-1$, $d$-dimensional cells in $\Sigma_f$ correspond to $(n-d)$-simplexes in $\Delta_f$.
  \item These correspondences are inclusion-reversing.
  \item If $n=2$, then $1$-dimensional cell in $\Sigma_f$ is orthogonal to the corresponding $1$-simplex in $\Delta_f$ (see Figure \ref{fdual}).
\end{itemize}
For a cell $\sigma \in \Delta_f$, the corresponding cell in $\Sigma_f$ is given by $\{P\in \mathbb{R}^n\ |\ f(P)=c_{\bi}+\bi\cdot P \text{ for any lattice point $\bi$ contained in $\sigma$}\}$, where $\cdot$ denotes the Euclidean inner product.

\end{theorem}

\begin{figure}[H]
\centering
\begin{tikzpicture}
\coordinate (L1) at (2,0);
\coordinate (L2) at (2.5,0);
\coordinate (L3) at (3,0);
\coordinate (L4) at (2,0.5);
\coordinate (L5) at (2.5,0.5);
\coordinate (L6) at (2,1);
\coordinate (L7) at (-1,1);
\coordinate (L8) at (-0.5,1);
\coordinate (L9) at (0,0.5);
\coordinate (L10) at (0.5,0.5);
\coordinate (L11) at (1,1);
\coordinate (L12) at (0,0);
\coordinate (L13) at (0.5,0);
\coordinate (L14) at (-1,1.5);
\coordinate (L15) at (-0.5,1.5);
\coordinate (L16) at (0,2);
\coordinate (L17) at (-0.5,0);
\coordinate (L18) at (-1,0.5);
\draw (L1)--(L3)--(L6)--cycle;
\draw (L4)--(L5)--(L2);
\draw (L11)--(L10)--(L18);
\draw (L16)--(L15)--(L17);
\draw (L14)--(L15);
\draw (L10)--(L13);
\coordinate (A1) at (7,0);
\coordinate (A2) at (7.5,0);
\coordinate (A3) at (8,0);
\coordinate (A4) at (7,0.5);
\coordinate (A5) at (7.5,0.5);
\coordinate (A6) at (7,1);
\coordinate (A7) at (4,0.5);
\coordinate (A8) at (4.5,0.5);
\coordinate (A9) at (5,1);
\coordinate (A10) at (5.5,1);
\coordinate (A11) at (6,1.5);
\coordinate (A12) at (4.5,0);
\coordinate (A13) at (5.5,0);
\coordinate (A14) at (4,1.5);
\coordinate (A15) at (5,1.5);
\coordinate (A16) at (5.5,2);
\draw (A1)--(A3);
\draw (A3)--(A6);
\draw (A1)--(A6);
\draw (A4)--(A5)--(A2)--cycle;
\draw (A7)--(A8)--(A9)--(A10)--(A11);
\draw (A14)--(A15)--(A16);
\draw (A9)--(A15);
\draw (A8)--(A12);
\draw (A10)--(A13);
\coordinate [label=below:\text{$\Delta_f$}] (a) at (2.5,-0.25);
\coordinate [label=below:\text{$\Delta_g$}] (b) at (7.5,-0.25);
\coordinate [label=below:\text{$V(f)$}] (a) at (0,-0.25);
\coordinate [label=below:\text{$V(g)$}] (b) at (5,-0.25);
\coordinate [label=below:\text{$\leftrightarrow$}] (a) at (1.5,0.7);
\coordinate [label=below:\text{$\leftrightarrow$}] (b) at (6.5,0.7);
\end{tikzpicture}
\caption{Tropical plane curves and their dual subdivisions.}
\label{fdual}
\end{figure}
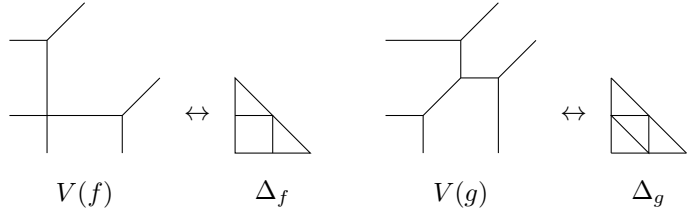

\begin{notation}
For the tropical plane curve $V(f)\subset \bR^2$ defined by a tropical polynomial $f$, we call a $0$-dimensional cell of $\Sigma_f$ a vertex of $V(f)$ and a $1$-dimensional cell of $\Sigma_f$ an edge of $V(f)$.
\end{notation}

For a tropical polynomial $f\in \bT[x^{\pm 1}, y^{\pm 1}]$, the dual subdivision $\Delta_f$ defines weights of the edges of the tropical plane curve $V(f)$.

\begin{definition}\label{weight}
For the tropical plane curve $V(f)\subset \bR^2$ defined by a tropical polynomial $f$, let $\sigma$ be an edge of $V(f)$.
We define the \textit{weight} $w_{\sigma}$ of $\sigma$ in $V(f)$ as the lattice length of the corresponding $1$-simplex of $\Delta_f$.
\end{definition}

From now on, when we say a tropical curve, we will refer to a polyhedral set together with weights on its edges.

\begin{notation}\label{v_P,L}
Let $\Gamma$ be a tropical plane curve, $P$ a vertex of $\Gamma$ and $L$ an edge of $\Gamma$ containing $P$.
Let $R$ be the ray which contains $L$ and has $P$ as its endpoint.
We denote by $\mathbf{v}_{P, L}$ the primitive vector that have the same direction as $R$.
\end{notation}

Tropical plane curves satisfy a balancing condition.

\begin{theorem}\label{MS3.3.2}
\cite[Theorem 3.3.2]{MS}
Let $\Gamma$ be a tropical plane curve and $P$ a vertex of $\Gamma$, $L_1, \dots, L_n$ the edges of $\Gamma$ containing $P$, and $w_{L_i}$ the weight of $L_i$ ($i=1, \dots, n$).
Then, we have
\[
\sum_{i=1}^nw_{L_i}\mathbf{v}_{P, L_i}=(0, \dots, 0)\in \bR^n.
\]
\end{theorem}

\begin{definition}[Divisor]
We define $\Divp$ as the set of all lines, rays and line segments in $\bR^2$, i.e., $\Divp=\{L\subset \bR^2\ |\ \text{$L$ is a line, a ray or a line segment}\}$ and
\begin{eqnarray*} 
\Divs:=\left\{ \sum_{L\in \Divp}\!\!\!w_LL\ \ (\text{formal sum})\ \middle| 
\begin{array}{l}
\text{$w_L\in \bZ$,}\\
\text{$\#\{L\in \Divp\ |\ w_L\neq 0\}<\infty$}
\end{array}
\right\}.
\end{eqnarray*}
For two elements $\sum_{L\in \Divp}w_LL\in \Divs$ and $\sum_{L\in \Divp}w'_LL\in \Divs$, we define the sum of them as $\sum_{L\in \Divp}(w_L+w'_L)L$ and denote it by $\sum_{L\in \Divp}w_LL+\sum_{L\in \Divp}w'_LL$.
For $L_1, L_2\in \Divp$ satisfying $\#(L_1\cap L_2)=1$ and $\exists L_3\in \Divp\ \text{s.t.}\ L_1\cup L_2=L_3$, and for each integer $n\in \bZ$, we identify $\sum_{L\in \Divp}w_LL$ and $\sum_{L\in \Divp \setminus \{L_1, L_2, L_3\}}w_LL+(w_{L_1}-n)L_1+(w_{L_2}-n)L_2+(w_{L_3}+n)L_3$.
This identification defines an equivalence relation on $\Divs$.
The quotient set of this equivalence relation is denoted by $\Div$.
The addition in $\Div$ is induced by the addition in $\Divs$.
Under this operation, $\Div$ forms an abelian group. The identity element of this group is denoted by $0$, and the inverse of $d\in \Div$ is denoted by $-d$.
\end{definition}

\begin{remark}
Since a tropical curve can be regarded as a sum of weighted edges, it can be regarded as an element of $\Div$.
\end{remark}

\begin{definition}
Let $f\in \bT[x^{\pm1}, y^{\pm1}]\setminus\{-\infty\}$.
When $V(f)$ is a tropical plane curve, we define $\div(f)$ to be the corresponding element of $\Div$. 
If $V(f)=\emptyset$, we define $\div(f)=0\in \Div$.
\end{definition}

\begin{remark}
For a tropical polynomial $f\in \bT[x^{\pm1}, y^{\pm1}]\setminus \{-\infty\}$,
\[
\text{$f$ is a unit} \Leftrightarrow V(f)=\emptyset \Leftrightarrow \div(f)=0.
\]
\end{remark}

\begin{lemma}\label{Jos4.6}
\cite[Lemma 4.6]{Jos}
For $f, g\in \bT[x^{\pm1}, y^{\pm1}]\setminus \{-\infty\}$, we have
\[
\div(f\odot g)=\div(f)+\div(g).
\]
\end{lemma}

\begin{notation}
Let $A, B\subset \bR^n$.
When $A$ is a translation of $B$, we write $A\sim B$.
When $A$ cannot be obtained from $B$ by any translation, we write $A\not\sim B$.
Note that $A\sim B$ if and only if $B\sim A$.
We denote the Minkowski sum of $A$ and $B$ by $A+B$.
\end{notation}

\begin{lemma}\label{div->Newt}
Let $f, g\in \bT[x^{\pm1}, y^{\pm1}]\setminus\{-\infty\}$ satisfy $\div(f)=\div(g)$.
Then, $\Delta_f$ is a translation of $\Delta_g$.
In particular, we have $\Newt(f)\sim \Newt(g)$.
\end{lemma}

\begin{proof}
Since $\div(f)=\div(g)$, we have $V(f)=V(g)$ as sets, and each edge of $V(f)$ and $V(g)$ has the same weight.
Then, by the duality theorem (Theorem \ref{dualitytheorem}), $\Delta_f$ is a translation of $\Delta_g$, and hence, we have $\Newt(f)\sim \Newt(g)$.
\end{proof}

\begin{lemma}\cite[Equation (2.3.3)]{MS}\label{fgnoNewt}
For $f, g\in \bT[x_1^{\pm1}, \dots, x_n^{\pm1}]\setminus \{-\infty\}$, we have
\[
\Newt(f\odot g)=\Newt(f)+\Newt(g).
\]
\end{lemma}

\begin{definition}
Let $f\in \bT[x^{\pm1}, y^{\pm1}]\setminus \{-\infty\}$.
The \textit{recession divisor} of $f$ is $\div(f_0)$ defined by $f_0=\bigoplus_{(i, j)\in \Newt(f)}0x^iy^j\in \bB[x^{\pm1}, y^{\pm1}]$.
We denote $\div(f_0)$ by $\recdiv(f)$.
\end{definition}

Let $f\in \bT[x^{\pm1}, y^{\pm1}]\setminus \{-\infty\}$.
Then, for a sufficiently large $d\in \bR_{>0}$, the set $V(f) \setminus \{(x, y)\in \bR^2\ |\ x^2+y^2<d\}$ is a disjoint union of finite rays $R_1, \dots, R_n$.
For each $i=1, \dots, n$, let $w_i$ be the weights of $R_i$.
For $i=1, \dots, n$, let $L_i$ be the translation of $R_i$ whose endpoint is at the origin.
Each edge $L_i$ is assigned weight $w_i$.
By the duality theorem (Theorem \ref{dualitytheorem}), $\Newt(f)$ has the information about the directions and the weights of $R_1, \dots, R_n$, i.e., those of $L_1, \dots, L_n$.
Thus, when $g\in \bT[x^{\pm1}, y^{\pm1}]\setminus \{-\infty\}$ satisfies $\recdiv(f)=\recdiv(g)$, we have $\Newt(f)\sim \Newt(g)$.

\begin{lemma}\label{Newtrecdiv}
Let $f, g\in \bT[x^{\pm1}, y^{\pm1}]\setminus \{-\infty\}$.
Then, $\recdiv(f)=\recdiv(g)$ if and only if $\Newt(f)\sim \Newt(g)$.
\end{lemma}

\begin{proof}
It is sufficient to show that $\recdiv(f)=\recdiv(g)$ if $\Newt(f)\sim \Newt(g)$.
Assume that $\Newt(f)\sim \Newt(g)$.
If $\Newt(f)$ and $\Newt(g)$ are points, then $f$ and $g$ are units, and hence, $\recdiv(f)=\recdiv(g)=0$.
If $\Newt(f)$ and $\Newt(g)$ are line segments, then there is a line $L$ through the origin such that $\recdiv(f)=\recdiv(g)=\ell L$, where $\ell$ is the lattice length of $\Newt(f)$.
In the following, we assume that $\Newt(f)$ and $\Newt(g)$ are two-dimensional sets.
Let $f_0=\bigoplus_{(i, j)\in \Newt(f)}0x^iy^j, g_0=\bigoplus_{(i, j)\in \Newt(g)}0x^iy^j\in \bB[x^{\pm1}, y^{\pm1}]$, $\recdiv(f)=1L_1+\dots +1L_n$ and $\recdiv(g)=1R_1+\dots +1R_m$, where $L_1, \dots, L_n$ and $R_1, \dots, R_m$ are rays whose endpoints are the origin.
Since $\Newt(f)$ has the information about the directions and the weights of $L_1, \dots, L_n$ and $\Newt(g)$ has the information about those of $R_1, \dots, R_m$.
Noting that the only vertex of $V(f_0)$ and $V(g_0)$ is the origin, we have $\recdiv(f)=\div(f_0)=\div(g_0)=\recdiv(g)$.
\end{proof}

\begin{definition}
Let $f\in \bT[x_1^{\pm1}, \dots, x_n^{\pm1}]\setminus \{-\infty\}$.
We say that $f$ is \textit{irreducible} if $f$ is not a unit and whenever $f$ is written in the form $f=g\odot h$, where $g$ and $h$ are tropical polynomials, then either $g$ or $h$ is a unit.
\end{definition}

In \cite{TW}, two notions of complexity for a tropical polynomial are considered.

\begin{definition}\cite[Definition 1]{TW}\label{Comp}
Let $f, f_1, \dots, f_n$ be tropical polynomials.
The \textit{monomial complexity} of $f$, denoted by $\mComp(f)$, is the number of linear regions of the function defined by $f$.
If $f=f_1\odot \dots \odot f_n$, the \textit{factorization complexity} of $f$ with respect to $(f_1, \dots, f_n)$, denoted by $\fComp(f_1, \dots, f_n)$, is defined as $\sum_{i=1}^{n}\mComp(f_i)-(n-1)$.
\end{definition}

In \cite{TW}, the following open question (Open Question 2) is presented: ``Do tropical polynomials have unique irreducible factorization with respect to the factorization complexity?''
We give a counterexample.

\begin{proposition}\label{thm_main3}
The tropical polynomial
\begin{eqnarray*}
&&x^2y^3\oplus xy^4\oplus x^2y^2\oplus xy^3\oplus x^2y\oplus xy^2\oplus y^3\oplus xy\oplus y^2\oplus x\oplus y
\end{eqnarray*}
has at least two different irreducible factorizations with the minimal factorization complexity.
\end{proposition}

\begin{proof}
Let $f=x^2y^3\oplus xy^4\oplus x^2y^2\oplus xy^3\oplus x^2y\oplus xy^2\oplus y^3\oplus xy\oplus y^2\oplus x\oplus y$.
Let $v_1=(1, 0)$, $v_2=(1, 1)$, $v_3=(-1, 1)$, $v_4=(-1, 0)$, $v_5=(-1, -1)$, $v_6=(1, -1)$ and $L_i:=\{r v_i\ |\ r\geq 0\}$ ($i=1, \dots, 6$).
Then, we have
\[
\div(f)=2L_1+1L_2+1L_3+2L_4+1L_5+1L_6.
\]

The tropical polynomial $f$ has at least two decompositions $f=(xy^2\oplus xy\oplus x\oplus y)(xy\oplus y^2\oplus y\oplus 0)$ and $f=(xy\oplus 0)(x\oplus y)(y\oplus 0)(y\oplus 0)$.
The corresponding factorizations of $\Newt(f)$ are as in Figure \ref{fComp}.
The factorization complexities are the following:
\begin{eqnarray*}
\fComp(xy^2\oplus xy\oplus x\oplus y, xy\oplus y^2\oplus y\oplus 0)&=&\fComp(xy\oplus 0, x\oplus y, y\oplus 0, y\oplus 0)\\
&=&5.
\end{eqnarray*}

These are factorizations into irreducible polynomials.
For example, we show that $xy^2\oplus xy\oplus x\oplus y$ is irreducible.
Assume that $xy^2\oplus xy\oplus x\oplus y=g\odot h$.
By Lemma \ref{Jos4.6}, we have $\div(g)+\div(h)=\div(xy^2\oplus xy\oplus x\oplus y)=2L_1+1L_3+1L_5$.
By symmetry, we may assume that $L_3\subset V(g)$.
Then, by the balancing condition, $L_3\cup L_5\subset V(g)$, and hence, by the balancing condition, we have $\div(g)=2L_1+1L_3+1L_5$.
Therefore, $\div(h)=0$, i.e., $h$ is a unit.
Thus, $xy^2\oplus xy\oplus x\oplus y$ is irreducible.

Therefore, it suffices to show that if irreducible polynomials 
$f_1, \dots, f_n$ satisfy $f=f_1\odot \dots\odot f_n$, then $\fComp(f_1, \dots, f_n)\geq 5$.
Assume that there are non-unit tropical polynomials $f_1, \dots, f_n$ ($n\geq 2$) satisfying $f=f_1\odot \dots \odot f_n$.
Then, we have $\div(f)=\div(f_1)+\dots +\div(f_n)$ by Lemma \ref{Jos4.6}.
By symmetry, we may assume that $L_2\subset V(f_1)$. Suppose first that $L_3\subset V(f_1)$.
Then, by the balancing condition, $L_5\cup L_6\subset V(f_1)$.
Thus, $\div(f_1)$ can only be $1L_1+1L_2+1L_3+1L_4+1L_5+1L_6$ or $1L_2+1L_3+1L_5+1L_6$.
In this case, the only possible tuples of divisors $(\div(f_1), \dots, \div(f_n))$ are
\begin{eqnarray*}
&(1)&(1L_1+1L_2+1L_3+1L_4+1L_5+1L_6,\ 1L_1+1L_4),\\ &(2)&(1L_2+1L_3+1L_5+1L_6,\ 2L_1+2L_4),\\
&(3)&(1L_2+1L_3+1L_5+1L_6,\ 1L_1+1L_4,\ 1L_1+1L_4).
\end{eqnarray*}

We next consider the case where $L_3\not\subset V(f_1)$.
The balancing condition then implies that $V(f_1)$ contains exactly one of $L_5$ and $L_6$.
If $L_5\subset V(f_1)$, the balancing condition for $V(f_1)$ shows that $\div(f_1)=1L_2+1L_5$, $\div(f_1)=1L_1+1L_2+1L_4+1L_5$ or $\div(f_1)=2L_1+1L_2+2L_4+1L_5$.
In this case, the only possible tuples of divisors $(\div(f_1), \dots, \div(f_n))$ are
\begin{eqnarray*}
&(4)&(1L_2+1L_5,\ 2L_1+1L_3+2L_4+1L_6),\\
&(5)&(1L_2+1L_5,\ 1L_1+1L_3+1L_4+1L_6,\ 1L_1+1L_4),\\
&(6)&(1L_2+1L_5,\ 1L_3+1L_6,\ 2L_1+2L_4),\\
&(7)&(1L_2+1L_5,\ 1L_3+1L_6,\ 1L_1+1L_4,\ 1L_1+1L_4),\\
&(8)&(1L_1+1L_2+1L_4+1L_5,\ 1L_1+1L_3+1L_4+1L_6),\\ &(9)&(1L_1+1L_2+1L_4+1L_5,\ 1L_1+1L_4,\ 1L_3+1L_6),\\
&(10)&(2L_1+1L_2+2L_4+1L_5,\ 1L_3+1L_6).
\end{eqnarray*}
Here, we identify those that differ only by a permutation of $\div(f_2), \dots, \div(f_n)$.

On the other hand, if $L_6\subset V(f_1)$, then the balancing condition for $V(f_1)$ gives $\div(f_1)=1L_2+2L_4+1L_6$.
In this case, the only possible tuple of divisors $(\div(f_1), \dots, \div(f_n))$ is
\begin{eqnarray*}
&(11)&(1L_2+2L_4+1L_6,\ 2L_1+1L_3+1L_5).
\end{eqnarray*}

If necessary, by multiplying $f_1, \dots, f_n$ by units $u_1, \dots, u_n$ ($u_1\odot \dots \odot u_n=0$), the decompositions of $\Newt(f)$ as Minkowski sums of $\Newt(f_1), \dots, \Newt(f_n)$ corresponding to (1)--(11) are as follows.
\begin{eqnarray*}
&(1)&\Conv(\{(1, 0), (2, 1), (2, 2), (1, 3), (0, 2), (0, 1)\})+\Conv(\{(0, 0), (0, 1)\}),\\
&(2)&\Conv(\{(1, 0), (2, 1), (1, 2), (0, 1)\})+\Conv(\{(0, 0), (0, 2)\}),\\
&(3)&\Conv(\{(1, 0), (2, 1), (1, 2), (0, 1)\})+\Conv(\{(0, 0), (0, 1)\})\\
&&+\Conv(\{(0, 0), (0, 1)\}),\\
&(4)&\Conv(\{(0, 1), (1, 0)\})+\Conv(\{(0, 0), (1, 1), (1, 3), (0, 2)\}),\\
&(5)&\Conv(\{(0, 1), (1, 0)\})+\Conv(\{(0, 0), (1, 1), (1, 2), (0, 1)\})\\
&&+\Conv(\{(0, 0), (0, 1)\}),\\
&(6)&\Conv(\{(0, 1), (1, 0)\})+\Conv(\{(0, 0), (1, 1)\})+\Conv(\{(0, 0), (0, 2)\}),\\
&(7)&\Conv(\{(0, 1), (1, 0)\})+\Conv(\{(0, 0), (1, 1)\})+\Conv(\{(0, 0), (0, 1)\})\\
&&+\Conv(\{(0, 0), (0, 1)\}),\\
&(8)&\Conv(\{(0, 1), (1, 0), (1, 1), (0, 2)\})+\Conv(\{(0, 0), (1, 1), (1, 2), (0, 1)\}),\\
&(9)&\Conv(\{(0, 1), (1, 0), (1, 1), (0, 2)\})+\Conv(\{(0, 0), (0, 1)\})\\
&&+\Conv(\{(0, 0), (1, 1)\}),\\
&(10)&\Conv(\{(1, 0), (1, 2), (0, 3), (0, 1)\})+\Conv(\{(0, 0), (1, 1)\}),\\
&(11)&\Conv(\{(0, 0), (1, 1), (0, 2)\})+\Conv(\{(0, 1), (1, 0), (1, 2)\}).
\end{eqnarray*}

For each of (1)--(11), the possible values of $\fComp(f_1, \dots, f_n)$ are calculated as follows.
\begin{eqnarray*}
&(1)&\fComp(f_1, f_2)\geq 6+2-1=7,\\
&(2)&\fComp(f_1, f_2)\geq 4+2-1=5,\\
&(3)&\fComp(f_1, f_2, f_3)\geq 4+2+2-2=6,\\
&(4)&\fComp(f_1, f_2)\geq 2+4-1=5,\\
&(5)&\fComp(f_1, f_2, f_3)=2+4+2-2=6,\\
&(6)&\fComp(f_1, f_2, f_3)\geq 2+2+2-2=4,\\
&(7)&\fComp(f_1, f_2, f_3, f_4)=2+2+2+2-3=5,\\
&(8)&\fComp(f_1, f_2)= 4+4-1=7,\\
&(9)&\fComp(f_1, f_2, f_3)=4+2+2-2=6,\\
&(10)&\fComp(f_1, f_2)\geq 4+2-1=5,\\
&(11)&\fComp(f_1, f_2)\geq 3+3-1=5.
\end{eqnarray*}

Thus, it suffices to show that the case (6) does not occur in the irreducible decomposition of $f$.
In the case (6), if necessary, by multiplying $f_1, f_2, f_3$ by units $u_1$, $u_2$, $u_3$ ($u_1\odot u_2\odot u_3=0$), we can assume that
\[
f=(y\oplus 0)\odot (xy\oplus 0)\odot(y^2\oplus \alpha y\oplus 0).
\]
Assume that $y^2\oplus \alpha y\oplus 0$ is irreducible.
If $\alpha >0$, then $(y^2\oplus \alpha y\oplus 0)(0, 0)=\alpha+0>0+2\times 0=0$, and hence, by the definition of $V(y^2\oplus \alpha y\oplus 0)$, we have $(0, 0)\notin V(y^2\oplus \alpha y\oplus 0)$.
This contradicts the fact that $\div(y^2\oplus \alpha y\oplus 0)=2L_1+2L_4$.
Thus, we have $\alpha \in \bR_{\leq 0}\cup \{-\infty\}$, and hence, we have
\begin{eqnarray*}
&&(y\oplus x)\odot (xy\oplus 0)\odot(y^2\oplus \alpha y\oplus 0)\\
&&=(x^2y\oplus xy^2\oplus x\oplus y)\odot(y^2\oplus \alpha y\oplus 0)\\
&&=x^2y^3\oplus xy^4\oplus \alpha x^2y^2\oplus \alpha xy^3\oplus x^2y\oplus xy^2\oplus y^3\oplus \alpha xy\oplus \alpha y^2\oplus x\oplus y\\
&&\neq f.
\end{eqnarray*}
Therefore, the case (6) does not occur in the irreducible decomposition of $f$.
\end{proof}

\begin{figure}[H]
\centering
\begin{tikzpicture}
\coordinate (A1) at (0.5,-1);
\coordinate (A2) at (1,-0.5);
\coordinate (A3) at (1,0.5);
\coordinate (A4) at (0.5,1);
\coordinate (A5) at (0,0.5);
\coordinate (A6) at (0,-0.5);
\coordinate (A7) at (2,0);
\coordinate (A8) at (2.5,-0.5);
\coordinate (A9) at (2.5,0.5);
\coordinate (A10) at (3.5,-0.5);
\coordinate (A11) at (4,0);
\coordinate (A12) at (3.5,0.5);
\coordinate (A13) at (5,-0.25);
\coordinate (A14) at (5.5,0.25);
\coordinate (A15) at (6.5,0.25);
\coordinate (A16) at (7,-0.25);
\coordinate (A17) at (8,-0.25);
\coordinate (A18) at (8,0.25);
\coordinate (A19) at (9,-0.25);
\coordinate (A20) at (9,0.25);
\draw (A1)--(A2)--(A3)--(A4)--(A5)--(A6)--cycle;
\draw (A7)--(A8)--(A9)--cycle;
\draw (A10)--(A11)--(A12)--cycle;
\draw (A13)--(A14);
\draw (A15)--(A16);
\draw (A17)--(A18);
\draw (A19)--(A20);
\coordinate [label=below:\text{$=$}] (a) at (1.5,0.16);
\coordinate [label=below:\text{$+$}] (b) at (3,0.23);
\coordinate [label=below:\text{$=$}] (c) at (4.5,0.16);
\coordinate [label=below:\text{$+$}] (d) at (6,0.23);
\coordinate [label=below:\text{$+$}] (b) at (7.5,0.23);
\coordinate [label=below:\text{$+$}] (c) at (8.5,0.23);
\end{tikzpicture}
\caption{Two different factorizations of a hexagon.}
\label{fComp}
\end{figure}
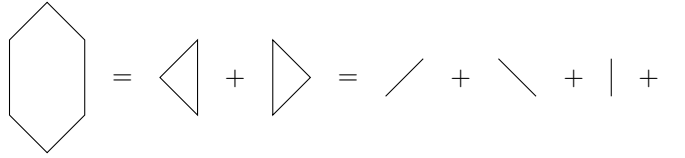

\section{Volume of a pair of tropical polynomials}

When a tropical rational function $\varphi$ is given, we can express $\varphi$ as $\varphi=f\oslash g$, where $f$ and $g\neq -\infty$ are tropical polynomials.
In \cite{TW}, a preorder on the set of pairs of tropical polynomials is defined as follows.
We define $(f_1, g_1)\leq_{\mComp}(f_2, g_2)$ if $\mComp(f_1)\leq \mComp(f_2)$ and $\mComp(g_1)\leq \mComp(g_2)$.
When $f_1$, $g_1$, $f_2$ and $g_2$ are expressed as
\begin{eqnarray*}
f_1&=&f_{11}\odot \dots \odot f_{1k},\\
g_1&=&g_{11}\odot \dots \odot g_{1\ell},\\
f_2&=&f_{21}\odot \dots \odot f_{2m},\\
g_2&=&g_{21}\odot \dots \odot g_{2n},
\end{eqnarray*}
we define $(f_{11}, \dots, f_{1k}; g_{11}, \dots, g_{1\ell})\leq_{\fComp}(f_{21}, \dots, f_{2m}; g_{21}, \dots, g_{2n})$ if we have
\begin{eqnarray*}
\fComp(f_{11}, \dots, f_{1k})&\leq&\fComp(f_{21}, \dots, f_{2m}),\\
\fComp(g_{11}, \dots, g_{1\ell})&\leq&\fComp(g_{21}, \dots, g_{2n}).
\end{eqnarray*}
However, there are pairs $(f_1, g_1)$ and $(f_2, g_2)$ satisfying $\mComp(f_1)> \mComp(f_2)$ and $\mComp(g_1)<\mComp(g_2)$.
This is also true in the case of the factorization complexity.
To overcome the above difficulty, we define a volume $\vol(f, g)$ for a pair $(f, g)$.
This satisties $\vol(f_1, g_1)\leq \vol(f_2, g_2)$ or $\vol(f_1, g_1)\geq \vol(f_2, g_2)$ for any pairs $(f_1, g_1)$ and $(f_2, g_2)$.

\begin{theorem}[The duality theorem for tropical rational functions]\label{f/g}
Let $\varphi$ be a tropical rational function and $f\in\bT[x_1^{\pm1}, \dots, x_n^{\pm1}]$ and $g\in\bT[x_1^{\pm1}, \dots, x_n^{\pm1}]\setminus \{-\infty\}$ tropical polynomials satisfying $\varphi=f\oslash g$.
Then, we have the following:
\begin{eqnarray*}
V(f\oplus (x_{n+1}\odot g))&=&\{(\mathbf{x}, \varphi(\mathbf{x}))\in \bR^{n+1}\ |\ \mathbf{x} \in \bR^n, \varphi(\mathbf{x})\neq -\infty\}\\
&&\cup \{(\mathbf{x}, x_{n+1})\in \bR^{n+1}\ |\ \mathbf{x}\in V(f), x_{n+1}< \varphi(\mathbf{x})\}\\
&&\cup \{(\mathbf{x}, x_{n+1})\in \bR^{n+1}\ |\ \mathbf{x}\in V(g), x_{n+1}> \varphi(\mathbf{x})\}.
\end{eqnarray*}
\end{theorem}

\begin{proof}
When $\varphi=-\infty$, the numerator $f$ is $-\infty$.
Then, both of the sets $\{(\mathbf{x}, \varphi(\mathbf{x}))\in \bR^{n+1}\ |\ \mathbf{x} \in \bR^n, \varphi(\mathbf{x})\neq -\infty\}$ and $\{(\mathbf{x}, x_{n+1})\in \bR^{n+1}\ |\ \mathbf{x}\in V(f), x_{n+1}< \varphi(\mathbf{x})\}$ are empty and the following holds:
\begin{eqnarray*}
V(f\oplus (x_{n+1}\odot g))=V(x_{n+1}\odot g)=\{(\mathbf{x}, x_{n+1})\in \bR^{n+1}\ |\ \mathbf{x}\in V(g), x_{n+1}\in \bR\}.
\end{eqnarray*}

Next, assume that $\varphi\neq -\infty$.
Then, we have $f\in\bT[x_1^{\pm1}, \dots, x_n^{\pm1}]\setminus \{-\infty\}$.
Let $P=(p_1, \dots, p_{n+1})\in V(f\oplus (x_{n+1}\odot g))$.
Assume that $P\notin \{(\mathbf{x}, \varphi(\mathbf{x}))\in \bR^{n+1}\ |\ \mathbf{x} \in \bR^n, \varphi(\mathbf{x})\neq -\infty\}$.
First, we consider the case when $p_{n+1}>\varphi (p_1, \dots, p_n)$.
Then, we have
\[
p_{n+1}+g(p_1, \dots, p_n)>f(p_1, \dots, p_n),
\]
and hence, we have
\[
(x_{n+1}\odot g)(P)>f(P).
\]
Combined with the assumption that $P\in V(f\oplus (x_{n+1}\odot g))$, it follows that at least two terms of $x_{n+1}\odot g$ take the value $(x_{n+1}\odot g)(P)$ at the point $P$, and hence, we have $(p_1, \dots, p_n)\in V(g)$.
Thus, $P\in \{(\mathbf{x}, x_{n+1})\in \bR^{n+1}\ |\ \mathbf{x}\in V(g), x_{n+1}> \varphi(\mathbf{x})\}$ in this case.
In the same way, when $p_{n+1}<\varphi (p_1, \dots, p_n)$, we have $P\in \{(\mathbf{x}, x_{n+1})\in \bR^{n+1}\ |\ \mathbf{x}\in V(f), x_{n+1}< \varphi(\mathbf{x})\}$.
Thus, we have
\begin{eqnarray*}
V(f\oplus (x_{n+1}\odot g))&\subset &\{(\mathbf{x}, \varphi(\mathbf{x}))\in \bR^{n+1}\ |\ \mathbf{x} \in \bR^n, \varphi(\mathbf{x})\neq -\infty\}\\
&&\cup \{(\mathbf{x}, x_{n+1})\in \bR^{n+1}\ |\ \mathbf{x}\in V(f), x_{n+1}< \varphi(\mathbf{x})\}\\
&&\cup \{(\mathbf{x}, x_{n+1})\in \bR^{n+1}\ |\ \mathbf{x}\in V(g), x_{n+1}> \varphi(\mathbf{x})\}.
\end{eqnarray*}

Finally, we show the reverse inclusion under the assumption that $\varphi\neq -\infty$.
Let $(a_1, \dots, a_{n+1})\in \{(\mathbf{x}, \varphi(\mathbf{x}))\in \bR^{n+1}\ |\ \mathbf{x} \in \bR^n, \varphi(\mathbf{x})\neq -\infty\}$.
Then, we have
\begin{eqnarray*}
f(a_1, \dots, a_n)&=&f(a_1, \dots, a_n)-g(a_1, \dots, a_n)+g(a_1, \dots, a_n)\\
&=&\varphi(a_1, \dots, a_n)+g(a_1, \dots, a_n)\\
&=&a_{n+1}+g(a_1, \dots, a_n)\\
&=&(x_{n+1}\odot g)(a_1, \dots, a_{n+1}),
\end{eqnarray*}
and hence, noting Definition \ref{troppoly}, there exist a term $t_1$ of $f$ and a term $t_2$ of $x_{n+1}\odot g$ satisfying the equation
\[
t_1(a_1, \dots, a_{n+1})=t_2(a_1, \dots, a_{n+1})= (f\oplus (x_{n+1}\odot g))(a_1, \dots, a_{n+1}).
\]
It follows that $(a_1, \dots, a_{n+1})\in V(f\oplus (x_{n+1}\odot g))$.
From the above, we have $\{(\mathbf{x}, \varphi(\mathbf{x}))\in \bR^{n+1}\ |\ \mathbf{x} \in \bR^n, \varphi(\mathbf{x})\neq -\infty\}\subset V(f\oplus (x_{n+1}\odot g))$.
Next, let $(b_1, \dots, b_{n+1})\in \{(\mathbf{x}, x_{n+1})\in \bR^{n+1}\ |\ \mathbf{x}\in V(f), x_{n+1}< \varphi(\mathbf{x})\}$.
Then, we have
\begin{eqnarray*}
&&b_{n+1}<f(b_1, \dots, b_n)-g(b_1, \dots, b_n),\\
&\text{i.e., }&f(b_1, \dots, b_n)>b_{n+1}+g(b_1, \dots, b_n)=(x_{n+1}\odot g)(b_1, \dots, b_{n+1}).
\end{eqnarray*}
Combined with the fact that $(b_1, \dots, b_n)\in V(f)$, it follows that $(b_1, \dots, b_{n+1})\in V(f\oplus (x_{n+1}\odot g))$.
Thus, we have $\{(\mathbf{x}, x_{n+1})\in \bR^{n+1}\ |\ \mathbf{x}\in V(f), x_{n+1}< \varphi(\mathbf{x})\}\subset V(f\oplus (x_{n+1}\odot g))$.
In the same way, we can show that $\{(\mathbf{x}, x_{n+1})\in \bR^{n+1}\ |\ \mathbf{x}\in V(g), x_{n+1}> \varphi(\mathbf{x})\}\subset V(f\oplus (x_{n+1}\odot g))$.
\end{proof}

Thus, $\Delta_{f\oplus(x_{n+1}\odot g)}$ is, in a sense, dual to the graph of $f\oslash g$.
As a corollary of the above theorem, we have the following.

\begin{corollary}\label{jogenobunkatsu}
Let $f\in\bT[x_1^{\pm1}, \dots, x_n^{\pm1}]$ and $g\in\bT[x_1^{\pm1}, \dots, x_n^{\pm1}]\setminus \{-\infty\}$ be tropical polynomials.
Then, $\Delta_{f\oplus (x_{n+1}\odot g)}$ induces the subdivision of $\{(x, y, 0)\ |\ (x, y)\in \Newt(f)\}$, and this can be identified with $\Delta_f$.
$\Delta_{f\oplus (x_{n+1}\odot g)}$ induces the subdivision of $\{(x, y, 1)\ |\ (x, y)\in \Newt(g)\}$, and this can be identified with $\Delta_g$.
\end{corollary}


\begin{proposition}\label{ichiisei}
Let $f_1, g_1, f_2, g_2\in \bT[x^{\pm1}, y^{\pm1}]\setminus \{-\infty\}$ satisfy $\div(f_1)=\div(f_2)$, $\div(g_1)=\div(g_2)$ and $f_1\oslash g_1=f_2\oslash g_2$ as functions.
Then, $\Delta_{f_1\oplus(z\odot g_1)}$ is a translation of $\Delta_{f_2\oplus(z\odot g_2)}$.
\end{proposition}

\begin{proof}
Let $P_i=\Newt(f_i\oplus (z\odot g_i))$ ($i=1, 2$).
For $i=1, 2$, we call $\{(x, y, 1)\ |\ (x, y)\in P_i\}$ the upper face of $P_i$ and $\{(x, y, 0)\ |\ (x, y)\in P_i\}$ the lower face of $P_i$.
First, since $\div(f_1)=\div(f_2)$ and $\div(g_1)=\div(g_2)$, by Lemma \ref{div->Newt}, $\Delta_{f_1}$ is a translation of $\Delta_{f_2}$, and $\Delta_{g_1}$ is a translation of $\Delta_{g_2}$.
Therefore, by Corollary \ref{jogenobunkatsu}, the subdivision of the upper face of $P_1$ is a translation of that of $P_2$.
The same holds for the subdivisions of their lower faces.
Moreover, by Theorem \ref{f/g}, among the edges in $\Delta_{f_1\oplus(z\odot g_1)}$, those not contained in either the upper or the lower face correspond to the graph of $f_1\oslash g_1$.
These edges, in turn, correspond to the edges in $\Delta_{f_2\oplus(z\odot g_2)}$ that are not contained in either the upper or the lower face.
It follows that $P_1\sim P_2$, and $\Delta_{f_1\oplus(z\odot g_1)}$ is a translation of $\Delta_{f_2\oplus(z\odot g_2)}$.
\end{proof}

\begin{definition}
Let $f, g\in \bT[x_1^{\pm1}, \dots, x_n^{\pm1}]$ be tropical polynomials.
We define the \textit{volume} of the pair $(f, g)$ as the $(n+1)$-dimensional Lebesgue volume of $\Newt(f\oplus (x_{n+1}\odot g))$, and denote it by $\vol(f, g)$.
\end{definition}

\begin{example}
Let $\varphi(x)$ be the tropical rational function defined by $\varphi(x)=f_1(x)\oslash g_1(x)$, where $f_1(x)=x\oplus 0$ and $g_1(x)=x\oplus 1$ are tropical polynomials in $\bT[x^{\pm1}]$.
We can also express $\varphi(x)$ as $\varphi(x)=f_2(x)\oslash g_2(x)$, where $f_2(x)=(-2)x^2\oplus x\oplus 0$ and $g_2(x)=(-2)x^2\oplus x\oplus 1$.
Here, the volume of the pair $(f_1, g_1)$ is the area of the Newton polytope of $f_1\oplus (y\odot g_1)$ and this is $1$ (see Figure \ref{f-g}).
On the other hand, we have $\vol(f_2, g_2)=2$.
\end{example}

\begin{figure}[H]
\centering
\begin{tikzpicture}
\coordinate (L1) at (5,-0.5);
\coordinate (L2) at (6,-0.5);
\coordinate (L3) at (6.5,0);
\coordinate (L4) at (7.5,0);
\coordinate (L5) at (6,-1);
\coordinate (L6) at (6.5,1);
\draw[->,>=stealth,gray] (5,0)--(7.5,0) node[right]{$x$};
\draw[->,>=stealth,gray] (6,-1)--(6,1) node[right]{$y$};
\draw (6,0) node[gray, above left]{O};
\draw[thick] (L1)--(L2)--(L3)--(L4);
\draw[thick] (L2)--(L5);
\draw[thick] (L3)--(L6);
\coordinate (R1) at (8.5,-0.5);
\coordinate (R2) at (9.5,-0.5);
\coordinate (R3) at (10,0);
\coordinate (R4) at (11,0);
\coordinate (R5) at (9.5,-1);
\coordinate (R6) at (10,1);
\coordinate (R7) at (10.5,-1);
\coordinate (R8) at (10.5,1);
\draw[->,>=stealth,gray] (8.5,0)--(11,0) node[right]{$x$};
\draw[->,>=stealth,gray] (9.5,-1)--(9.5,1) node[right]{$y$};
\draw (9.5,0) node[gray, above left]{O};
\draw[thick] (R1)--(R2)--(R3)--(R4);
\draw[thick] (R2)--(R5);
\draw[thick] (R3)--(R6);
\draw[thick] (R7)--(R8);
\coordinate (A1) at (1.25,-0.5);
\coordinate (A2) at (1.75,-0.5);
\coordinate (A3) at (1.25,0);
\coordinate (A4) at (1.75,0);
\coordinate (B1) at (3,-0.5);
\coordinate (B2) at (3.5,-0.5);
\coordinate (B3) at (4,-0.5);
\coordinate (B4) at (3,0);
\coordinate (B5) at (3.5,0);
\coordinate (B6) at (4,0);
\draw (A1)--(A2)--(A4)--(A3)--cycle;
\draw (A2)--(A3);
\coordinate [label=below:\text{$V(f_1\oplus (y\odot g_1))$}] (a) at (6.25,-1.25);
\coordinate [label=below:\text{$V(f_2\oplus (y\odot g_2))$}] (b) at (9.75,-1.25);
\coordinate [label=below:\text{$\Delta_{f_1\oplus (y\odot g_1)}$}] (c) at (1.5,-1.25);
\draw (B1)--(B3)--(B6)--(B4)--cycle;
\draw (B4)--(B2)--(B5);
\coordinate [label=below:\text{$\Delta_{f_2\oplus (y\odot g_2)}$}] (c) at (3.5,-1.25);
\end{tikzpicture}
\caption{Two representations of the same tropical rational function.}
\label{f-g}
\end{figure}
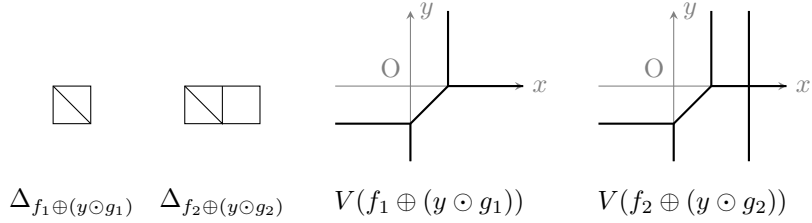

\begin{proposition}
Let $f, g, h\in \bT[x_1^{\pm1}, \dots, x_n^{\pm1}]$ be tropical polynomials.
Assume that $g, h\neq -\infty$ and the dimension of the affine span of $\Newt(f\oplus (x_{n+1}\odot g))$ is $n+1$.
Then, we have $\vol(f\odot h, g\odot h)\geq \vol(f, g)$ and the equality holds if and only if $h$ is a unit, i.e., a monomial.
\end{proposition}

\begin{proof}
First, if $u\in \bT[x_1^{\pm1}, \dots, x_n^{\pm1}]$ is a unit, i.e., a monomial, then $\Newt((f\odot u)\oplus (x_{n+1}\odot (g\odot u)))\sim \Newt(f\oplus (x_{n+1}\odot g))$, and hence, $\vol(f\odot u, g\odot u)= \vol(f, g)$.

Noting that for a unit $v$, we have
\[
\vol(f\odot h\odot v, g\odot h\odot v)=\vol(f\odot h, g\odot h),
\]
we can assume that the constant term of $h$ is not $-\infty$ by multiplying a unit.
Then, we have $\Newt(f)\subset \Newt(f\odot h)$ and $\Newt(g)\subset \Newt(g\odot h)$, and hence, we have $\vol(f\odot h, g\odot h)\geq \vol(f, g)$.
If $h$ is not a monomial, then, by Lemma \ref{fgnoNewt}, $\Newt(f)\subsetneq \Newt(f)+\Newt(h)=\Newt(f\odot h)$ and $\Newt(g)\subsetneq \Newt(g\odot h)$, and hence, $\vol(f\odot h, g\odot h)> \vol(f, g)$.
\end{proof}

When a tropical rational function $\varphi$ on $\bR^n$ is given, is there a unique expression $\varphi=f\oslash g$ with the minimum volume $\vol(f, g)$?
The answer is yes when $n=1$ (Theorem \ref{thm_main1}).

\begin{lemma}\label{f=g}
Let $f(x), g(x)\in \bT[x^{\pm1}]$ be tropical polynomials.
If $f=g$ as functions, then we have $\Delta^{\uparrow}_{f}=\Delta^{\uparrow}_{g}$ and $\Delta_{f}=\Delta_{g}$.
\end{lemma}

\begin{proof}
The tropical rational function $f(x)\oslash 0$ is equal to $g(x)\oslash 0$ as functions.
In particular, the graphs of $f(x)\oslash 0$ and $g(x)\oslash 0$ are equal.
Combined with Theorem \ref{f/g}, by considering the slopes of the graphs, it follows that $\Delta_{f(x)\oplus y}=\Delta_{g(x)\oplus y}$ and $V(f(x)\oplus y)=V(g(x)\oplus y)$.
Noting that the slopes of the graphs are determined by the coefficients of $f$ and $g$, we have $\Delta^{\uparrow}_{f}=\Delta^{\uparrow}_{g}$, and hence, $\Delta_{f}=\Delta_{g}$.
\end{proof}

\begin{theorem}\label{thm_main1}
For any tropical rational function $\varphi(x) \neq -\infty$ on $\bR$, there is an expression $\varphi(x)=f(x)\oslash g(x)$ with the minimum volume.
If $f'(x)$ and $g'(x)$ also satisty $\varphi(x)=f'(x)\oslash g'(x)$ and $\vol(f', g')=\vol(f, g)$, then the dual subdivision of $f'\oplus(y\odot g')$ is a translation of the dual subdivision of $f\oplus(y\odot g)$.
In other words, the dual subdivision of $f\oplus(y\odot g)$ is unique up to translation.
\end{theorem}

\begin{proof}
By Pick's theorem, the area of the Newton polytope of finite lattice points is a multiple of $1/2$.
Since $\varphi \neq -\infty$, there exist $f, g\in \bT[x^{\pm1}]\setminus \{-\infty\}$ such that $\varphi=f\oslash g$ with the minimum volume.
By the fundamental theorem of algebra in the tropical algebra (see \cite[p.5]{MS}), the functions $f$ and $g$ can be written uniquely as tropical products of linear functions:
\begin{eqnarray*}
f&=&\alpha_1x^{i_0}(x\oplus a_1)^{i_1}\odot \dots \odot (x\oplus a_n)^{i_n},\\
g&=&\alpha_2x^{j_0}(x\oplus b_1)^{j_1}\odot \dots \odot (x\oplus b_m)^{j_m},
\end{eqnarray*}
where $\alpha_1, \alpha_2\in \bR$, $i_0, j_0\in \bZ$, $a_1< \dots< a_n\in \bR$, $b_1< \dots< b_m\in \bR$ and $i_1,\dots, i_n, j_1, \dots, j_m\in \bZ_{>0}$.
Note that by Lemma \ref{f=g} and Pick's theorem, we have
\[
\vol(f, g)=\frac{\sum_{s=1}^{n}i_s+1+\sum_{s=1}^{m}j_s+1}{2}-1.
\]
Also note that for a real number $a\in \bR$ and tropical polynomials $h(x)$ and $h'(x)$, we have $h\oslash h'=(h\odot(x\oplus a))\oslash (h'\odot(x\oplus a))$ as functions.
Combined with the fact that $\vol(f, g)$ is the minimum volume, it follows that $\{a_1, \dots, a_n\}\cap \{b_1, \dots, b_m\}=\emptyset$.

Assume that $f'(x)$ and $g'(x)$ also satisty $\varphi(x)=f'(x)\oslash g'(x)$ and $\vol(f', g')=\vol(f, g)$.
Then, the functions $f'$ and $g'$ can be written uniquely as tropical products of linear functions:
\begin{eqnarray*}
f'&=&\alpha_3x^{k_0}(x\oplus c_1)^{k_1}\odot \dots \odot (x\oplus c_u)^{k_u},\\
g'&=&\alpha_4x^{\ell_0}(x\oplus d_1)^{\ell_1}\odot \dots \odot (x\oplus d_v)^{\ell_v},
\end{eqnarray*}
where $\alpha_3, \alpha_4\in \bR$, $k_0, \ell_0\in \bZ$, $c_1< \dots< c_u\in \bR$, $d_1< \dots< d_v\in \bR$, $k_1,\dots, k_u$, $\ell_1, \dots, \ell_v\in \bZ_{>0}$ and $\{c_1, \dots, c_u\}\cap \{d_1, \dots, d_v\}=\emptyset$.
Since $f\oslash g=f'\oslash g'$ as functions, the functions $f\odot g'$ and $f'\odot g$ are equal.
Therefore, the functions $(\alpha_1\odot \alpha_4)x^{i_0+\ell_0}(x\oplus a_1)^{i_1}\odot \dots \odot (x\oplus a_n)^{i_n}\odot (x\oplus d_1)^{\ell_1}\odot \dots \odot (x\oplus d_v)^{\ell_v}$ and $(\alpha_3\odot \alpha_2)x^{k_0+j_0}(x\oplus c_1)^{k_1}\odot \dots \odot (x\oplus c_u)^{k_u}\odot (x\oplus b_1)^{j_1}\odot \dots \odot (x\oplus b_m)^{j_m}$ are equal.
These expressions as tropical products of linear functions are unique.
Combined with the fact that $\{a_1, \dots, a_n\}\cap \{b_1, \dots, b_m\}=\emptyset$ and $\{c_1, \dots, c_u\}\cap \{d_1, \dots, d_v\}=\emptyset$, it follows that $\alpha_1-\alpha_3=\alpha_2-\alpha_4$, $i_0-k_0=j_0-\ell_0$, $n=u$, $a_s=c_s$, $i_s=k_s$ ($1\leq s \leq n$), $m=v$ and $b_t=d_t$, $j_t=\ell_t$ ($1\leq t \leq m$).
Combined with Lemma \ref{f=g}, it follows that $\Delta^{\uparrow}_f$ (resp. $\Delta^{\uparrow}_g$) is equal to the translation of $\Delta^{\uparrow}_{f'}$ (resp. $\Delta^{\uparrow}_{g'}$) by the vector $(i_0-k_0, \alpha_1-\alpha_3)\in \bR^2$, and hence, $\Delta^{\uparrow}_{f\oplus (y\odot g)}$ is equal to the translation of $\Delta^{\uparrow}_{f'\oplus (y\odot g')}$ by the vector $(i_0-k_0, 0, \alpha_1-\alpha_3)\in \bR^3$.
Therefore, $\Delta_{f\oplus (y\odot g)}$ is equal to the translation of $\Delta_{f'\oplus (y\odot g')}$ by the vector $(i_0-k_0, 0)$.
\end{proof}

On the other hand, there can be two different minimum volume expressions of a tropical rational function on $\bR^2$.
Before stating Proposition \ref{thm_main2}, we first prove several theorems concerning the volumes of a three-dimensional convex lattice polytopes.

\begin{definition}
Let $v\in \bR^n\setminus \{0\}$ and $a\in \bR$.
A \textit{closed half-space} in $\bR^n$ is the set of the form
\[
\{x\in \bR^n\ |\ v\cdot x\geq a\},
\]
where $\cdot$ denotes the Euclidean inner product.
\end{definition}

\begin{definition}
A (convex) \textit{polytope} $P$ in $\bR^n$ is a bounded set that can be represented as the intersection of finitely many closed half-spaces and has a nonempty interior, i.e., it contains an $n$-dimensional open ball.
A \textit{lattice polytope} is a polytope whose vertices all lie in the integer lattice $\bZ^n$.
When $n=2$, a polytope is called a \textit{polygon}.
\end{definition}

\begin{definition}
Given a polytope $P\subset \bR^n$ and $k\in \bZ_{>0}$, we define $kP$ as follows:
\[
kP=\{kx\ |\ x\in P\}.
\]
\end{definition}

\begin{notation}
Let $A\subset \bR^2$ and $P\subset \bR^3$ be Lebesgue measurable sets.
We denote their respective Lebesgue volumes by $\Area(A)$ and $\vol(P)$. 
\end{notation}

\begin{notation}
Let $P\subset \bR^n$ be a polytope.
We denote the number of lattice points in the interior of $P$ by $i(P)$, the number of lattice points on the boundary of $P$ by $b(P)$, and their sum by $\ell(P)$.
\end{notation}

There is a volume formula for three-dimensional lattice polytopes.

\begin{theorem}\cite{Ree}\label{Reeve}
Let $P$ be a lattice polytope in $\bR^3$.
Then, we have
\[
\vol(P)=\frac{i(2P)-2i(P)}{6}+\frac{b(2P)-2b(P)}{12}.
\]
\end{theorem}

\begin{lemma}\label{2A}
Let $A\subset \bR^2$ be a lattice polygon.
Then,
\[
\ell(2A)=4\Area(A)+b(A)+1.
\]
\end{lemma}

\begin{proof}
By Pick's theorem, we have
\[
4\Area(A)=\Area(2A)=i(2A)+\frac{b(2A)}{2}-1.
\]
Noting that $b(2A)$ is the sum of the lattice lengths of the boundary edges of $2A$, and hence, $b(2A)=2b(A)$, we have
\begin{eqnarray*}
\ell(2A)&=&i(2A)+b(2A)\\
&=&\left(i(2A)+\frac{b(2A)}{2}-1\right)+\frac{b(2A)}{2}+1\\
&=&4\Area(A)+\frac{b(2A)}{2}+1\\
&=&4\Area(A)+b(A)+1.
\end{eqnarray*}
\end{proof}

\begin{theorem}\label{volthm}
Let $A$ and $B$ be lattice polygons and $P$ be the convexhull of $\{(x, y, 0)\in \bR^3\ |\ (x, y)\in A\}\cup\{(x, y, 1)\in \bR^3\ |\ (x, y)\in B\}$.
Then, we have
\[
6\vol(P)=\Area(A)+\Area(B)+\Area(A+B).
\]
\end{theorem}

\begin{proof}
First, note that the following hold:
\begin{eqnarray*}
\{(x, y)\ |\ (x, y, 1)\in 2P\}&=&\left\{(2x, 2y)\ \middle|\ \left(x, y, \frac{1}{2}\right)\in P\right\}\\
&=&\left\{\left(2\frac{x_0+x_1}{2}, 2\frac{y_0+y_1}{2}\right)\ \middle|\ (x_0, y_0)\in A, (x_1, y_1)\in B\right\}\\
&=&A+B,
\end{eqnarray*}
and hence, we have $i(2P)=i(A+B)$.
Noting that $i(P)=0$, we have
\begin{eqnarray*}
6\vol(P)&=&i(2P)-2i(P)+\frac{b(2P)}{2}-b(P)\ (\because \text{Theorem \ref{Reeve}})\\
&=&i(A+B)+\frac{b(A+B)+\ell(2A)+\ell(2B)}{2}-(\ell(A)+\ell(B))\\
&=&i(A+B)+\frac{b(A+B)}{2}+\frac{4\Area(A)+b(A)+4\Area(B)+b(B)+2}{2}\\
&&-\ell(A)-\ell(B)\ (\because \text{Lemma \ref{2A}})\\
&=&i(A+B)+\frac{b(A+B)}{2}+\frac{4\Area(A)+b(A)+4\Area(B)+b(B)+2}{2}\\
&&-\Area(A)-\frac{b(A)}{2}-1-\Area(B)-\frac{b(B)}{2}-1\ (\because \text{Pick's theorem})\\
&=&i(A+B)+\frac{b(A+B)}{2}-1+\Area(A)+\Area(B)\\
&=&\Area(A+B)+\Area(A)+\Area(B)\ (\because \text{Pick's theorem}).
\end{eqnarray*}
\end{proof}

In the following theorem, we do not assume that $A$ and $B$ are polygons.

\begin{theorem}\label{3voldainari}
Let $A$ and $B$ be the nonempty convex hulls of finite sets of lattice points in $\bZ^2$ and
\[
P:=\Conv\left(\{(x, y, 0)\in \bR^3\ |\ (x, y)\in A\}\cup\{(x, y, 1)\in \bR^3\ |\ (x, y)\in B\}\right).
\]
Then, we have
\[
\max\{\Area(A), \Area(B)\}+2\min\{\Area(A), \Area(B)\}\leq 3\vol(P).
\]
\end{theorem}

\begin{proof}
By symmetry, we may assume without loss of generality that $\Area(A)\geq \Area(B)$.
When $\Area(A)=\Area(B)=0$, we have
\[
\max\{\Area(A), \Area(B)\}+2\min\{\Area(A), \Area(B)\}=0\leq 3\vol(P).
\]
Next, consider the case where $\Area(A)>0$ and $\Area(B)=0$.
Let $p\in \{(x, y, 1)\in \bR^3\ |\ (x, y)\in B\}$.
Since $\Conv\left(\{p\}\cup \{(x, y, 0)\in \bR^3\ |\ (x, y)\in A\}\right)\subset P$, we have $\Area(A)\leq 3\vol(P)$.
Thus, in this case, we have $\Area(A)+2\Area(B)\leq 3\vol(P)$.

In the following, we assume that $A$ and $B$ are polygons.
Then, by Theorem \ref{volthm}, we have
\[
6\vol(P)=\Area(A)+\Area(B)+\Area(A+B).
\]
The Brunn-Minkowski inequality \cite{Gar} states that
\[
\sqrt{\Area(A+B)}\geq \sqrt{\Area(A)}+\sqrt{\Area(B)},
\]
and hence,
\[
\Area(A+B)\geq \Area(A)+\Area(B)+2\sqrt{\Area(A)\Area(B)}.
\]
Thus, we have
\begin{eqnarray*}
3\vol(P)&=&\frac{\Area(A)+\Area(B)+\Area(A+B)}{2}\\
&\geq& \Area(A)+\Area(B)+\sqrt{\Area(A)\Area(B)}\\
&\geq& \Area(A)+\Area(B)+\frac{2\Area(A)\Area(B)}{\Area(A)+\Area(B)}\mbox{ ($\because$ AM-GM inequality)}\\
&\geq& \Area(A)+\Area(B)+\frac{2\Area(A)\Area(B)}{2\Area(A)}\\
&=&\Area(A)+2\Area(B).
\end{eqnarray*}
\end{proof}

\begin{theorem}\label{koushitenjogen}
Let $A_0$ and $A_1$ be the nonempty convex hulls of finite sets of lattice points in $\bZ^2$ and $P$ be the convex hull of $\{(x, y, 0)\in \bR^3\ |\ (x, y)\in A_0\}\cup\{(x, y, 1)\in \bR^3\ |\ (x, y)\in A_1\}$.
Assume that if either $A_0$ or $A_1$ has area zero, it contains at most two lattice points.
Let $A_M$ denote the one of $A_0$ and $A_1$ with the larger area, and $A_m$ the one with the smaller area.
Then, we have
\[
\ell(A_M)\leq 6\vol(P)+2
\]
and
\[
\ell(A_m)\leq 3\vol(P)+2.
\]
\end{theorem}

\begin{proof}
When $\Area(A_M)=0$, by the assumption, we have
\[
\max\{\ell(A_M), \ell(A_m)\}\leq 2\leq 3\vol(P)+2\leq 6\vol(P)+2.
\]
Next, consider the case where $\Area(A_M)>0$.
Then, by Pick's theorem and Theorem \ref{3voldainari}, we have
\[
\ell(A_M)=2\Area(A_M)+2-i(A_M)\leq 2\Area(A_M)+2\leq 6\vol(P)+2.
\]
If $\Area(A_m)=0$, then $\ell(A_m)\leq 2\leq 3\vol(P)+2$ holds.
If $\Area(A_m)>0$, then, by Pick's theorem and Theorem \ref{3voldainari}, we have
\[
\ell(A_m)=2\Area(A_m)+2-i(A_m)\leq 2\Area(A_m)+2\leq 3\vol(P)+2.
\]
\end{proof}

We now proceed to prove a lemma that will be used in the proof of Proposition \ref{thm_main2}.

\begin{lemma}\label{area2}
Let $\Delta_1$ be the convex hull of the three points $(0, 0)$, $(1, 0)$ and $(0, 1)$, $\Delta_2$ the convex hull of the three points $(1, 0)$, $(0, 1)$ and $(1, 1)$, and $\Delta$ the nonempty convex hull of finitely many lattice points in $\bZ^2$.
Assume that $\Delta$ is not a point.
For $i=1, 2$, the area of $\Delta_i+\Delta$ is $3/2$ if $\Delta$ is a translation of one of the following:
\[
\Conv (\{(0, 0), (1, 0)\}),\ \Conv (\{(0, 0), (0, 1)\}),\ \Conv (\{(1, 0), (0, 1)\})
\]
and is bigger than or equal to $2$ in the other cases.
\end{lemma}

\begin{proof}
Let us consider the case $i=1$.
We can deal with the case $i=2$ in the same way.
Let us consider the case when $\Delta\sim \Conv(\{(0, 0), (a, b)\})$ $(a, b\in \bZ^2, a\geq 0)$.
We may also assume $b> 0$ if $a=0$.
If $a>1$, then a translation of $\Delta_1+ \Delta$ contains the parallelogram $\Conv (\{(0, 0), (0, 1), (a, 1+b), (a, b)\})$, and hence, $\Area(\Delta_1+ \Delta)\geq 2$.
In the same way, if $|b|>1$, then the area of $\Delta_1+ \Delta$ is bigger than or equal to $2$.
Thus, it is sufficient to consider the case when $a\leq 1$ and $|b|\leq 1$, i.e., $(a, b)=(1, -1), (1, 0), (1, 1), (0, 1)$.
It is clear that $\Area(\Delta_1+ \Delta)=3/2$ when $(a, b)=(1, -1), (1, 0), (0, 1)$, and $\Area(\Delta_1+ \Delta)=5/2$ when $(a, b)=(1, 1)$.
Let us consider the case when $\Delta$ is a two dimensional set.
For a line segment $L$ contained in $\Delta$, we have $\Area(\Delta_1+\Delta)>\Area(\Delta_1+L)\geq 3/2$, and hence, $\Area(\Delta_1+\Delta)\geq 2$.
\end{proof}

\begin{definition}
Let $\varphi$ be a tropical rational function on $\bR^2$.
Then, we define $\cF(\varphi)$ as follows:
\[
\cF(\varphi)=\{(f, g)\in \bT[x^{\pm1}, y^{\pm1}]^2\ |\ g\neq-\infty \mbox{ and }f\oslash g=\varphi \mbox{ as functions}\}.
\]
\end{definition}

\begin{proposition}\label{thm_main2}
There is a tropical rational function on $\bR^2$, which has at least two different positive minimum volume expressions.
\end{proposition}

\begin{proof}
Let $f_1=xy\oplus (-1)y^2\oplus x\oplus y\oplus 0$ and $g_1=(-1)xy^2\oplus xy\oplus (-1)y^2\oplus x\oplus y$.
Then, the tropical curves $V(f_1)$ and $V(g_1)$, and $\div(f_1)-\div(g_1)$ are as in Figure \ref{f1g1}.
On the other hand, the tropical polynomials $f_2=x^2\oplus xy\oplus (-1)y^2\oplus x\oplus (-1)y$ and $g_2=(-1)x^2y\oplus (-1)xy^2\oplus x^2\oplus xy\oplus (-1)y^2$ satisfy $f_1\oslash g_1=f_2\oslash g_2$ as functions, and $\Area(\Newt(f_1)+\Newt(g_1))=\Area(\Newt(f_2)+\Newt(g_2))=7$ (see Figures \ref{f1g1} and \ref{f2g2}).
By Theorem \ref{volthm}, we have $\vol(f_1, g_1)=\vol(f_2, g_2)=5/3$.
Note that $\div(f_1)-\div(g_1) (=\div(f_2)-\div(g_2))$ contains the three rays $R_1:=\{(t, t+1)\ |\ t\geq 0\}$, $R_2:=\{(0, -t)\ |\ t\geq 0\}$ and $R_3:=\{(-t, 0)\ |\ t\geq 0\}$ with coefficients $1$.

Let $(f, g)\in \cF(f_1\oslash g_1)$.
Then, $\div(f)$ contains the three rays $R_1$, $R_2$ and $R_3$ with coefficients at least $1$.
Let $L_1, \dots, L_n$ be rays satisfying $\recdiv(f)=1L_1+\dots +1L_n$.
Here, $L_1, \dots, L_n$ are not necessarily distinct.
We may assume that for each $i=1, 2, 3$, a translation of $L_i$ is equal to $R_i$.
If $n=3$, then, by Lemma \ref{Newtrecdiv}, $\Newt(f)\sim \Delta_1:=\Conv(\{(0, 0), (1, 0), (0, 1)\})$, and hence, $V(f)\sim V(x\oplus y\oplus 0)$.
However, a translation of $V(x\oplus y\oplus 0)$ cannot contain $R_1\cup R_2\cup R_3$.
Therefore, we have $n\geq 4$.

Let $f_0:=\bigoplus_{(i, j)\in \Newt(f)}0x^iy^j,\ g_0:=\bigoplus_{(i, j)\in \Newt(g)}0x^iy^j\in \bB[x^{\pm1}, y^{\pm1}]$.
Let $v_1, \dots, v_n$ be the primitive direction vectors of the rays $L_1, \dots, L_n$.
Then, $V(f_0)$ satisfies the balancing condition:
\[
1v_1+\dots +1v_n=(0, 0).
\]
Since the three rays $L_1$, $L_2$ and $L_3$ satisfy the balancing condition $1v_1+1v_2+1v_3=(0, 0)$, the other rays $L_4, \dots, L_n$ also do, i.e.,
\[
1v_4+\dots +1v_n=(1v_1+\dots +1v_n)-(1v_1+1v_2+1v_3)=(0, 0).
\]
By \cite[Proposition 2.4]{Mik1}, there is an $h\in \bT[x^{\pm1}, y^{\pm1}]$ satisfying $\div(h)=1L_4+\dots +1L_n$.
By Lemma \ref{Jos4.6}, we have
\[
\div(f_0)=\div(x\oplus y\oplus 0)+\div(h)=\div((x\oplus y\oplus 0)\odot h).
\]
Combined with Lemmas \ref{div->Newt} and \ref{fgnoNewt}, it follows that
\[
\Newt(f)=\Newt(f_0)\sim \Newt((x\oplus y\oplus 0)\odot h)=\Newt(x\oplus y\oplus 0)+\Newt(h).
\]

Let $R'_1:=\{(t, 1)\ |\ t\geq 0\}$, $R'_2:=\{(0, t+1)\ |\ t\geq 0\}$, $L'_1:=\{(t, 0)\ |\ t\geq 0\}$, $L'_2:=\{(0, t)\ |\ t\geq 0\}$ and $L'_3:=\{(-t, -t)\ |\ t\geq 0\}$.
Then, we have $\div(f)-\div(g)=1R_1+1R_2+1R_3-1R'_1-1R'_2-1L'_3$, and hence,
\[
\recdiv(f)-\recdiv(g)=1L_1+1L_2+1L_3-1L'_1-1L'_2-1L'_3.
\]
Thus, we have $\recdiv(g)=1L'_1+1L'_2+1L'_3+1L_4+\dots +1L_n$ and
\[
\Newt(g)\sim \Conv(\{(1, 0), (0, 1), (1, 1)\})+\Newt(h).
\]

Let $\Delta:=\Conv(\{(0, 0), (2, 0), (1, 1), (0, 1)\})$.
Then, $\Delta$ contains five lattice points.
There are ten possible line segments determined by choosing two of these lattice points, six of which lie on the boundary of $\Delta$.
Thus, there are four possible ways to add a line segment with lattice points as its endpoints to $\Delta$.
These can be seen in Figure \ref{f8}.
A subdivision of $\Delta$ is determined by choosing some of these four line segments.
Since the line segments can intersect only at their endpoints, it follows that there are only eight possible subdivisions, as shown in Figure \ref{f8}.
None of the dual tropical curves can contain the three rays $R_1$, $R_2$ and $R_3$, and hence, $\Newt(f)\not\sim \Delta$.
It follows that $\Newt(h)\not\sim \Conv(\{(0, 0), (1, 0)\})$.
Furthermore, if $\Newt(h)\not\sim \Conv(\{(0, 0), (0, 1)\})$ and $\Newt(h)\not\sim \Conv(\{(1, 0), (0, 1)\})$, then we have $\Newt(f), \Newt(g)\geq 2$ by Lemma \ref{area2}, and hence, $\vol(f, g)\geq (2+4)/3=2$ by Theorem \ref{3voldainari}.
Therefore, if $\vol(f, g)<2$, we have $\Newt(h)\sim \Conv(\{(0, 0), (0, 1)\})$ or $\Newt(h)\sim \Conv(\{(1, 0), (0, 1)\})$.
In both cases, Theorem \ref{volthm} shows that $\vol(f, g)=5/3$.
It follows that $\varphi=f_1\oslash g_1$ and $\varphi=f_2\oslash g_2$ are minimum volume expressions of $\varphi$.
\end{proof}

\begin{figure}[H]
\centering
\begin{tikzpicture}
\coordinate (L1) at (5.5,-1);
\coordinate (L2) at (6,-1);
\coordinate (L3) at (5.5,-0.5);
\coordinate (L4) at (6,-0.5);
\coordinate (L5) at (5.5,0);
\coordinate (L6) at (7.5,-1);
\coordinate (L7) at (7,-0.5);
\coordinate (L8) at (7.5,-0.5);
\coordinate (L9) at (7,0);
\coordinate (L10) at (7.5,0);
\draw (L1)--(L2)--(L4)--(L5)--cycle;
\draw (L3)--(L4);
\draw (L6)--(L10)--(L9)--(L7)--cycle;
\draw (L7)--(L8);
\coordinate (L1) at (9,-0.5);
\coordinate (L2) at (9.5,-1);
\coordinate (L3) at (10,-1);
\coordinate (L4) at (10,0.5);
\coordinate (L5) at (9.5,1);
\coordinate (L6) at (9,1);
\coordinate (L7) at (9,0);
\coordinate (L8) at (9,0.5);
\coordinate (L9) at (9.5,-0.5);
\coordinate (L10) at (9.5,0);
\coordinate (L11) at (9.5,0.5);
\coordinate (L12) at (10,-0.5);
\coordinate (L13) at (10,0);
\coordinate (L14) at (10,0.5);
\draw (L1)--(L2)--(L3)--(L4)--(L5)--(L6)--cycle;
\foreach \t in {1,2,...,14} \fill[black] (L\t) circle (0.03);
\coordinate (A1) at (0,-1);
\coordinate (A2) at (0.5,-1);
\coordinate (A3) at (-0.5,-0.5);
\coordinate (A4) at (0.5,-0.5);
\coordinate (A5) at (1.5,-0.5);
\coordinate (A6) at (-0.5,0);
\coordinate (A7) at (0.5,0);
\coordinate (A8) at (1.5,0);
\coordinate (A9) at (0.5,0.5);
\coordinate (A10) at (1,0.5);
\coordinate (B1) at (3,-1);
\coordinate (B2) at (3.5,-1);
\coordinate (B3) at (2.5,-0.5);
\coordinate (B4) at (3.5,-0.5);
\coordinate (B5) at (4.5,-0.5);
\coordinate (B6) at (2.5,0);
\coordinate (B7) at (3.5,0);
\coordinate (B8) at (4.5,0);
\coordinate (B9) at (3.5,0.5);
\coordinate (B10) at (4,0.5);

\draw[thick, dotted] (A1)--(A4);
\draw[thick, dotted] (A8)--(A7)--(A9);
\draw[thick] (A7)--(A10);
\draw[thick] (A2)--(A4)--(A3);
\draw (A6)--(A7)--(A4)--(A5);
\coordinate [label=below:\text{$\Delta_{f_1}$}] (a) at (5.85,-1.25);
\coordinate [label=below:\text{$\Delta_{g_1}$}] (b) at (7.35,-1.25);
\coordinate [label=below:\text{$\Newt(f_1)+\Newt(g_1)$}] (b) at (9.85,-1.25);
\coordinate [label=below:\text{$V(f_1)$ and $V(g_1)$}] (c) at (0.5,-1.25);
\draw[thick, dotted] (B1)--(B4);
\draw[thick, dotted] (B8)--(B7)--(B9);
\draw[thick] (B7)--(B10);
\draw[thick] (B2)--(B4)--(B3);
\coordinate [label=below:\text{$\div(f_1)-\div(g_1)$}] (c) at (3.5,-1.25);
\end{tikzpicture}
\caption{Tropical curves $V(f_1)$ and $V(g_1)$.}
\label{f1g1}
\end{figure}
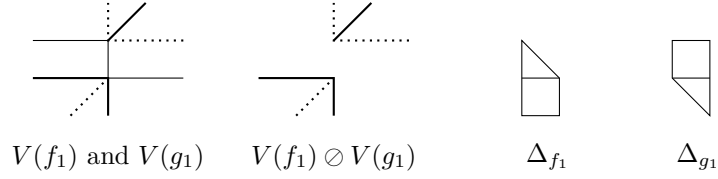

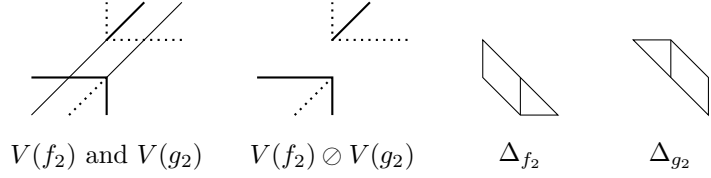
\begin{figure}[H]
\centering
\begin{tikzpicture}
\coordinate (L1) at (5.25,-0.5);
\coordinate (L2) at (5.75,-1);
\coordinate (L3) at (6.25,-1);
\coordinate (L4) at (5.75,-0.5);
\coordinate (L5) at (5.25,0);
\coordinate (L6) at (6.75,0);
\coordinate (L7) at (7.25,-0.5);
\coordinate (L8) at (7.75,-1);
\coordinate (L9) at (7.75,-0.5);
\coordinate (L10) at (7.25,0);
\draw (L1)--(L2)--(L3)--(L5)--cycle;
\draw (L2)--(L4);
\draw (L6)--(L8)--(L9)--(L10)--cycle;
\draw (L7)--(L10);
\coordinate (L1) at (9,0.5);
\coordinate (L2) at (10.5,-1);
\coordinate (L3) at (11,-1);
\coordinate (L4) at (11,-0.5);
\coordinate (L5) at (9.5,1);
\coordinate (L6) at (9,1);
\coordinate (L7) at (9.5,0);
\coordinate (L8) at (9.5,0.5);
\coordinate (L9) at (10,-0.5);
\coordinate (L10) at (10,0);
\coordinate (L11) at (10,0.5);
\coordinate (L12) at (10.5,-0.5);
\coordinate (L13) at (10.5,0);
\draw (L1)--(L2)--(L3)--(L4)--(L5)--(L6)--cycle;
\foreach \t in {1,2,...,13} \fill[black] (L\t) circle (0.03);
\coordinate (A1) at (0,-1);
\coordinate (A2) at (0.5,-1);
\coordinate (A3) at (-0.5,-0.5);
\coordinate (A4) at (0.5,-0.5);
\coordinate (A5) at (1.5,-0.5);
\coordinate (A6) at (-0.5,0);
\coordinate (A7) at (0.5,0);
\coordinate (A8) at (1.5,0);
\coordinate (A9) at (0.5,0.5);
\coordinate (A10) at (1,0.5);
\coordinate (A11) at (-0.5,-1);
\coordinate (A12) at (1.5,0.5);
\coordinate (B1) at (3,-1);
\coordinate (B2) at (3.5,-1);
\coordinate (B3) at (2.5,-0.5);
\coordinate (B4) at (3.5,-0.5);
\coordinate (B5) at (4.5,-0.5);
\coordinate (B6) at (2.5,0);
\coordinate (B7) at (3.5,0);
\coordinate (B8) at (4.5,0);
\coordinate (B9) at (3.5,0.5);
\coordinate (B10) at (4,0.5);
\draw[thick, dotted] (A1)--(A4);
\draw[thick, dotted] (A8)--(A7)--(A9);
\draw[thick] (A7)--(A10);
\draw[thick] (A2)--(A4)--(A3);
\draw (A7)--(A11);
\draw (A4)--(A12);
\coordinate [label=below:\text{$\Delta_{f_2}$}] (a) at (5.75,-1.25);
\coordinate [label=below:\text{$\Delta_{g_2}$}] (b) at (7.375,-1.25);
\coordinate [label=below:\text{$\Newt(f_2)+\Newt(g_2)$}] (b) at (9.9,-1.25);
\coordinate [label=below:\text{$V(f_2)$ and $V(g_2)$}] (c) at (0.5,-1.25);
\draw[thick, dotted] (B1)--(B4);
\draw[thick, dotted] (B8)--(B7)--(B9);
\draw[thick] (B7)--(B10);
\draw[thick] (B2)--(B4)--(B3);
\coordinate [label=below:\text{$\div(f_2)-\div(g_2)$}] (c) at (3.5,-1.25);
\end{tikzpicture}
\caption{Tropical curves $V(f_2)$ and $V(g_2)$.}
\label{f2g2}
\end{figure}

\begin{figure}[H]
\centering
\begin{tikzpicture}
\coordinate (L1) at (0.5,3.25);
\coordinate (L2) at (1,2.75);
\coordinate (L3) at (0.5,2.75);
\coordinate (L4) at (0,2.75);
\coordinate (L5) at (0,3.25);
\draw (L1)--(L2)--(L4)--(L5)--cycle;
\coordinate (A1) at (1.5,3);
\coordinate (A2) at (2,3);
\coordinate (A3) at (2.5,3.5);
\coordinate (A4) at (2,2.5);
\coordinate (A5) at (2,3.5);
\draw (A1)--(A2)--(A3);
\draw (A4)--(A5);

\coordinate (R1) at (5,3.25);
\coordinate (R2) at (5.5,2.75);
\coordinate (R3) at (5,2.75);
\coordinate (R4) at (4.5,2.75);
\coordinate (R5) at (4.5,3.25);
\draw (R1)--(R2)--(R4)--(R5)--cycle;
\draw (R1)--(R3);
\coordinate (B1) at (6,3);
\coordinate (B2) at (7,3);
\coordinate (B3) at (7.5,3.5);
\coordinate (B4) at (7,2.5);
\coordinate (B5) at (6.5,2.5);
\coordinate (B6) at (6.5,3.5);
\draw (B1)--(B2)--(B3);
\draw (B2)--(B4);
\draw (B5)--(B6);

\coordinate (L1) at (10,3.25);
\coordinate (L2) at (10.5,2.75);
\coordinate (L3) at (10,2.75);
\coordinate (L4) at (9.5,2.75);
\coordinate (L5) at (9.5,3.25);
\draw (L1)--(L2)--(L4)--(L5)--cycle;
\draw (L3)--(L5);
\coordinate (A1) at (11,3);
\coordinate (A2) at (11.5,3);
\coordinate (A3) at (12,3.5);
\coordinate (A4) at (11.75,2.5);
\coordinate (A5) at (11.75,3.5);
\coordinate (A6) at (11.5,2.5);
\draw (A1)--(A2)--(A3);
\draw (A4)--(A5);
\draw (A2)--(A6);

\coordinate (R1) at (0.5,1.75);
\coordinate (R2) at (1,1.25);
\coordinate (R3) at (0.5,1.25);
\coordinate (R4) at (0,1.25);
\coordinate (R5) at (0,1.75);
\draw (R1)--(R2)--(R4)--(R5)--cycle;
\draw (R1)--(R4);
\coordinate (B1) at (1.5,1.5);
\coordinate (B2) at (2,1.5);
\coordinate (B3) at (2.25,1.25);
\coordinate (B4) at (2.5,1.5);
\coordinate (B5) at (2.25,1);
\coordinate (B6) at (2,2);
\draw (B1)--(B2)--(B3)--(B4);
\draw (B3)--(B5);
\draw (B2)--(B6);

\coordinate (R1) at (5,1.75);
\coordinate (R2) at (5.5,1.25);
\coordinate (R3) at (5,1.25);
\coordinate (R4) at (4.5,1.25);
\coordinate (R5) at (4.5,1.75);
\draw (R1)--(R2)--(R4)--(R5)--cycle;
\draw (R2)--(R5);
\coordinate (B1) at (6,1.5);
\coordinate (B2) at (6.5,1.5);
\coordinate (B3) at (6.625,1.75);
\coordinate (B4) at (6.875,2);
\coordinate (B5) at (6.625,2);
\coordinate (B6) at (6.5,1);
\draw (B1)--(B2)--(B3)--(B4);
\draw (B3)--(B5);
\draw (B2)--(B6);

\coordinate (L1) at (10,1.75);
\coordinate (L2) at (10.5,1.25);
\coordinate (L3) at (10,1.25);
\coordinate (L4) at (9.5,1.25);
\coordinate (L5) at (9.5,1.75);
\draw (L1)--(L2)--(L4)--(L5)--cycle;
\draw (L3)--(L5);
\draw (L3)--(L1);
\coordinate (A1) at (11,1.5);
\coordinate (A2) at (11.25,1.5);
\coordinate (A3) at (11.5,1.75);
\coordinate (A4) at (11.75,1.75);
\coordinate (A5) at (12,2);
\coordinate (A6) at (11.75,1);
\coordinate (A7) at (11.5,2);
\coordinate (A8) at (11.25,1);
\draw (A1)--(A2)--(A3)--(A4)--(A5);
\draw (A4)--(A6);
\draw (A3)--(A7);
\draw (A2)--(A8);

\coordinate (R1) at (0.5,0.25);
\coordinate (R2) at (1,-0.25);
\coordinate (R3) at (0.5,-0.25);
\coordinate (R4) at (0,-0.25);
\coordinate (R5) at (0,0.25);
\draw (R1)--(R2)--(R4)--(R5)--cycle;
\draw (R1)--(R4);
\draw (R1)--(R3);
\coordinate (B1) at (1.5,0);
\coordinate (B2) at (2,0);
\coordinate (B3) at (2.25,-0.25);
\coordinate (B4) at (2.5,-0.25);
\coordinate (B5) at (2.75,0);
\coordinate (B6) at (2,0.5);
\coordinate (B7) at (2.25,-0.55);
\coordinate (B8) at (2.5,-0.55);
\draw (B1)--(B2)--(B3)--(B4)--(B5);
\draw (B2)--(B6);
\draw (B3)--(B7);
\draw (B4)--(B8);

\coordinate (R1) at (5,0.25);
\coordinate (R2) at (5.5,-0.25);
\coordinate (R3) at (5,-0.25);
\coordinate (R4) at (4.5,-0.25);
\coordinate (R5) at (4.5,0.25);
\draw (R1)--(R2)--(R4)--(R5)--cycle;
\draw (R2)--(R5);
\draw(R3)--(R5);
\coordinate (B1) at (6,-0.25);
\coordinate (B2) at (6.5,-0.25);
\coordinate (B3) at (6.75,0);
\coordinate (B4) at (6.875,0.25);
\coordinate (B5) at (7.125,0.5);
\coordinate (B6) at (6.5,-0.55);
\coordinate (B7) at (6.75,-0.55);
\coordinate (B8) at (6.875,0.5);
\draw (B1)--(B2)--(B3)--(B4)--(B5);
\draw (B2)--(B6);
\draw (B3)--(B7);
\draw (B4)--(B8);
\end{tikzpicture}
\caption{The eight subdivision of $\Delta$ and the shapes of their dual tropical curves.}
\label{f8}
\end{figure}
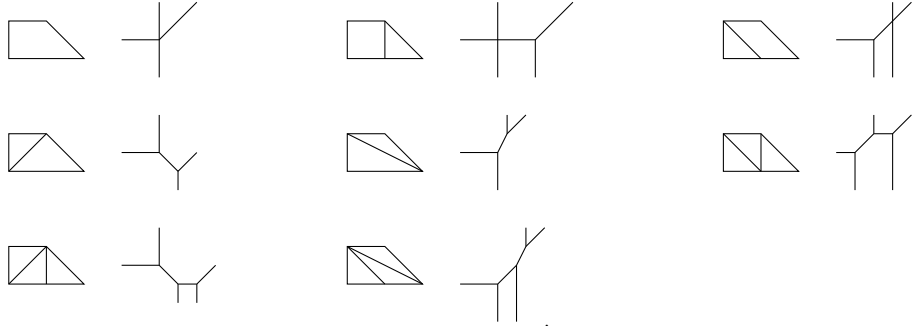

\begin{remark}
Let us consider the minimal factorization complexity for the tropical rational function in the above proof.
Let $f_1$, $f_2$, $g_1$ and $g_2$ be as in the proof of Proposition \ref{thm_main2}.
They can be factorized as follows:
\begin{eqnarray*}
f_1&=&(x\oplus (-1)y\oplus 0)\odot(y\oplus 0)=:f_{11}\odot f_{12},\\
f_2&=&(x\oplus y\oplus 0)\odot(x\oplus (-1)y)=:f_{21}\odot f_{22},\\
g_1&=&(xy\oplus x\oplus y)\odot((-1)y\oplus 0)=:g_{11}\odot g_{12},\\
g_2&=&((-1)xy\oplus x\oplus (-1)y)\odot(x\oplus y)=:g_{21}\odot g_{22}.
\end{eqnarray*}
Then, it holds that $\fComp(f_{11}, f_{12})=\dots =\fComp(g_{21}, g_{22})=4$.
Let $h_{11}, \dots, h_{1k}$, $h_{21}, \dots, h_{2\ell}$ be tropical polynomials satisfying $(h_{11}\odot \dots \odot h_{1k})\oslash (h_{21}\odot \dots \odot h_{2\ell})=f_1\oslash g_1$ as functions.
We can assume that each polynomial $h_{ij}$ satisfies $\mComp(h_{ij})\geq 2$ since monomials do not affect the factorization complexity.
Note that if $\mComp(h_{ij})=3$, then $V(h_{ij})$ is a union of three rays.
If $k=1$, then, by the proof of Proposition \ref{thm_main2}, we have $\fComp(h_{11})=\mComp(h_{11})\geq 4$.
If $k=2$ and $\fComp(h_{11}, h_{12})\leq 3$, then $\mComp(h_{11})=\mComp(h_{12})=2$, however it is impossible that the union of two lines contain the three rays $R_1$, $R_2$ and $R_3$ in the proof of Proposition \ref{thm_main2}.
Thus, we have $\fComp(h_{11}, h_{12})\geq 4$ if $k=2$.
If $k\geq 3$, then $\fComp(h_{11}, \dots, h_{1k})=\sum_{i=1}^{k}\mComp(h_{1i})-(k-1)\geq 2k-(k-1)=k+1\geq 4$.
Therefore, it holds that $\fComp(h_{11}, \dots, h_{1k})\geq 4$.
We can show that $\fComp(h_{21}, \dots, h_{2\ell})\geq 4$ in the same way.
Hence, $(f_{11}\odot f_{12})\oslash (g_{11}\odot g_{12})$ and $(f_{21}\odot f_{22})\oslash (g_{21}\odot g_{22})$ are also minimum in the sense of the factorization complexity.
\end{remark}

From the above remark, one would expect that a pair of tropical polynomials such that it is the minimum volume expression of a tropical rational function has the smallest factorization complexity, but unfortunately this is not true.

\begin{proposition}
There is a tropical rational function $\varphi$ that has the unique minimum volume expression $\varphi=f\oslash g$ up to translation of $\Delta_{f\oplus (z\odot g)}$, but none of the pairs of factorizations of $f$ and $g$ has the smallest factorization complexity.
\end{proposition}

\begin{proof}
Let $f_1=x^2\oplus xy\oplus y^2\oplus x\oplus y$, $g_1=xy\oplus x\oplus y$ and $\varphi=f_1\oslash g_1$.
Then, the tropical curves $V(f_1)$ and $V(g_1)$, and $\div(f_1)-\div(g_1)$ are as in Figure \ref{lf1g1}.
On the other hand, the tropical polynomials $f_2=x^2\oplus xy\oplus y^2\oplus x\oplus y\oplus 0$ and $g_2=xy\oplus x\oplus y\oplus 0$ satisfy $f_2\oslash g_2=\varphi$ as functions.
Here, $\Area(\Newt(f_1)+\Newt(g_1))=5$ (see Figure \ref{lf1g1}) and $\Area(\Newt(f_2)+\Newt(g_2))=7$ (see Figure \ref{lf2g2}).
By Theorem \ref{volthm}, we have
\[
\vol(f_1, g_1)=\frac{7}{6}<\frac{5}{3}=\vol(f_2, g_2).
\]

Let $(f, g)\in \cF(f_1\oslash g_1)$.
Then, $V(f)$ contains the three rays $R_1:=\{(t, t)\ |\ t\geq 0\}$, $R_2:=\{(-t, 0)\ |\ t\geq 0\}$ and $R_3:=\{(0, -t)\ |\ t\geq 0\}$.

Let $\Delta:=\Conv(\{(1, 0), (2, 0), (0, 2), (0, 1)\})$.
Then, $\Delta$ contains five lattice points.
There are ten possible line segments determined by choosing two of these lattice points, six of which lie on the boundary of $\Delta$.
Thus, there are four possible ways to add a line segment with lattice points as its endpoints to $\Delta$.
These can be seen in Figure \ref{f9}.
A subdivision of $\Delta$ is determined by choosing some of these four line segments.
Since the line segments can intersect only at their endpoints, it follows that there are only eight possible subdivisions, as shown in Figure \ref{f9}.
None of the dual tropical curves other than the first three can contain the three rays $R_1$, $R_2$ and $R_3$.
Noting that $\div(f)$ contains $R_1$ with coefficient at least $2$, if $\Newt(f)\sim \Delta$, then $\Delta_f$ is a translation of $\Delta_{f_1}$.
In this case, we have $\div(f)=\div(f_1)$, and hence, $\div(g)=\div(g_1)$.
In the following, we assume that $\Newt(f)\not\sim \Delta$.


Let $f_0:=\bigoplus_{(i, j)\in \Newt(f)}0x^iy^j,\ g_0:=\bigoplus_{(i, j)\in \Newt(g)}0x^iy^j\in \bB[x^{\pm1}, y^{\pm1}]$ and $\div(f_0)=1L_1+\dots +1L_n$, where $L_1, \dots, L_n$ are rays whose endpoints are the origin.
By renumbering, we can assume that $L_1=R_1$, $L_2=R_2$ and $L_3=R_3$.
Since $\div(f)$ contains $R_1$ with coefficient at least $2$, we have $n\geq 4$.
Let $v_1, \dots, v_n$ be the primitive direction vectors of the rays $L_1, \dots, L_n$.
Then, $V(f_0)$ satisfies the balancing condition:
\[
1v_1+\dots +1v_n=(0, 0).
\]
Since the three rays $L_1$, $L_2$ and $L_3$ satisfy the balancing condition $1v_1+1v_2+1v_3=(0, 0)$, the other rays $L_4, \dots, L_n$ also do, i.e.,
\[
1v_4+\dots +1v_n=(1v_1+\dots +1v_n)-(1v_1+1v_2+1v_3)=(0, 0).
\]
By \cite[Proposition 2.4]{Mik1}, there is an $h\in \bT[x^{\pm1}, y^{\pm1}]$ satisfying $\div(h)=1L_4+\dots +1L_n$.
By Lemma \ref{Jos4.6}, we have
\[
\div(f_0)=\div(x\oplus y\oplus 0)+\div(h)=\div((x\oplus y\oplus 0)\odot h).
\]
Combined with Lemmas \ref{div->Newt} and \ref{fgnoNewt}, it follows that
\[
\Newt(f)=\Newt(f_0)\sim \Newt((x\oplus y\oplus 0)\odot h)=\Newt(x\oplus y\oplus 0)+\Newt(h).
\]
Note that we have $\Newt(f)\not\sim \Conv(\{(0, 0), (1, 0), (1, 1), (0, 2)\})$ and $\Newt(f)\not\sim \Conv(\{(0, 0), (2, 0), (1, 1), (0, 1)\})$ since $\div(f)-\div(g)$ contains the ray $R_1$ with the coefficient $2$.
Noting the assumption that $\Newt(f)\not\sim \Delta$, we have $\Area(\Newt(f))\geq 2$ by Lemma \ref{area2}.

Since $V(g)$ contains two non-parallel rays $R'_1:=\{(t, 0)\ |\ t\geq 0\}$ and $R'_2:=\{(0, t)\ |\ t\geq 0\}$, $\Newt(g)$ is a two-dimensional set.
If $\Area(\Newt(g))=1/2$, then $\Newt(g)$ is the convex hull of three lattice points, and there is a ray $R'_3$ satisfying $\recdiv(g)=1R'_1+1R'_2+1R'_3$. 
By the balancing condition, $R'_3=\{(-t, -t)\ |\ t\geq 0\}$.
Thus, we have $\Newt(g)\sim \Conv(\{(1, 0), (0, 1), (1, 1)\})$, and hence, $\div(g)=1R'_1+1R'_2+1R'_3=\div(g_1)$.
Therefore, $\div(f)=\div(f_1)$ holds in this case.
Noting the assumption that $\Newt(f)\not\sim \Delta$, we may assume that $\Area(\Newt(g))\geq 1$.

We have $\Area(\Newt(f))\geq 2$ and $\Area(\Newt(g))\geq 1$.
By Theorem \ref{3voldainari}, it follows that
\[
\vol(f, g)\geq \frac{1}{3}\left(2+2\right)=\frac{4}{3}>\frac{7}{6}.
\]
Therefore, by Proposition \ref{ichiisei}, $f_1\oslash g_1$ is the only minimum volume expression of $\varphi$ up to translation of $\Delta_{f_1\oplus (z\odot g_1)}$.

Assume that there are non-unit tropical polynomials $h_1, \dots, h_n$ ($n\geq 2$) satisfying $f_1=h_1\odot \dots \odot h_n$.
Then, we have
\[
2R_1+1R_2+1R_3+1R'_3=\div(f_1)=\div(h_1)+\dots +\div(h_n)
\]
by Lemma \ref{Jos4.6}.
By symmetry, we may assume that $R_2\subset V(h_1)$.
Then, by the balancing condition, $R_1\subset V(h_1)$.
Furthermore, if $R'_3\subset V(h_1)$, then, by the balancing condition, $V(h_1)=V(f_1)$ holds. 
The same conclusion holds if the coefficient of $R_1$ in $\div(h_1)$ is equal to $2$.
Thus, $f_1$ has only two decompositions, namely the obvious one and $f_1=(x\oplus y\oplus 0)\odot (x\oplus y)$.
The polynomial $g_2$ can be decomposed as $g_2=(x\oplus 0)\odot (y\oplus 0)$.
Since we have
\begin{eqnarray*}
&&\fComp(x^2\oplus xy\oplus y^2\oplus x\oplus y)=4,\\
&&\fComp(x\oplus y\oplus 0, x\oplus y)=4,\\
&&\fComp(x^2\oplus xy\oplus y^2\oplus x\oplus y\oplus 0)=3,\\
&&\fComp(x\oplus 0, y\oplus 0)=3,
\end{eqnarray*}
it follows that
\begin{eqnarray*}
(f_2;\ x\oplus 0, y\oplus 0)&<_{\fComp}&(f_1;\ g_1),\\
(f_2;\ x\oplus 0, y\oplus 0)&<_{\fComp}&(x\oplus y\oplus 0, x\oplus y;\ g_1),
\end{eqnarray*}
i.e., the factorization complexity of the tropical rational function $\varphi$ with respect to $(f_2, (x\oplus 0)\odot (y\oplus 0))$ is strictly smaller than those of $(f_1, g_1)$ and $((x\oplus y\oplus 0)\odot (x\oplus y), g_1)$.
\end{proof}

\begin{figure}[H]
\centering
\begin{tikzpicture}
\coordinate (L1) at (5.5,0);
\coordinate (L2) at (6.5,-1);
\coordinate (L3) at (6,-1);
\coordinate (L4) at (5.5,-0.5);
\coordinate (L5) at (5.5,0);
\coordinate (L6) at (7.75,-1);
\coordinate (L7) at (7.25,-0.5);
\coordinate (L8) at (7.75,-0.5);
\draw (L1)--(L2)--(L3)--(L4)--cycle;
\draw (L6)--(L7)--(L8)--cycle;
\coordinate (L1) at (9,0);
\coordinate (L2) at (10,-1);
\coordinate (L3) at (10.5,-1);
\coordinate (L4) at (10.5,-0.5);
\coordinate (L5) at (9.5,0.5);
\coordinate (L6) at (9,0.5);
\coordinate (L7) at (9.5,-0.5);
\coordinate (L8) at (9.5,0);
\coordinate (L9) at (10,-0.5);
\coordinate (L10) at (10,0);
\draw (L1)--(L2)--(L3)--(L4)--(L5)--(L6)--cycle;
\foreach \t in {1,2,...,10} \fill[black] (L\t) circle (0.03);
\coordinate (A1) at (0,-0.5);
\coordinate (A2) at (0.5,-0.5);
\coordinate (A3) at (0.5,-1);
\coordinate (A4) at (0,-1);
\coordinate (A5) at (1,0);
\coordinate (A6) at (0.5,0);
\coordinate (A7) at (1,-0.5);
\coordinate (B1) at (3,-0.5);
\coordinate (B2) at (3.5,-0.5);
\coordinate (B3) at (3.5,-1);
\coordinate (B4) at (3,-1);
\coordinate (B5) at (4,0);
\coordinate (B6) at (3.5,0);
\coordinate (B7) at (4,-0.5);
\draw[thick, dotted] (A6)--(A2)--(A7);
\draw[thick] (0.5,-0.47)--(1,0.03);
\draw[thick] (0.525,-0.5)--(1.025,0);
\draw[thick] (A1)--(A2)--(A3);
\draw (A2)--(A4);
\coordinate [label=below:\text{$\Delta_{f_1}$}] (a) at (6,-1.25);
\coordinate [label=below:\text{$\Delta_{g_1}$}] (b) at (7.6,-1.25);
\coordinate [label=below:\text{$\Newt(f_1)+\Newt(g_1)$}] (b) at (10,-1.25);
\coordinate [label=below:\text{$V(f_1)$ and $V(g_1)$}] (c) at (0.5,-1.25);
\draw[thick, dotted] (B6)--(B2)--(B7);
\draw[thick] (3.5,-0.47)--(4,0.03);
\draw[thick] (3.525,-0.5)--(4.025,0);
\draw[thick] (B1)--(B2)--(B3);
\coordinate [label=below:\text{$\div(f_1)-\div(g_1)$}] (c) at (3.5,-1.25);
\end{tikzpicture}
\caption{Tropical curves $V(f_1)$ and $V(g_1)$.}
\label{lf1g1}
\end{figure}

\begin{figure}[H]
\centering
\begin{tikzpicture}
\coordinate (L1) at (5.5,0);
\coordinate (L2) at (6.5,-1);
\coordinate (L3) at (5.5,-1);
\coordinate (L4) at (6,-0.5);
\coordinate (L5) at (6,0);
\coordinate (L6) at (7.75,-1);
\coordinate (L7) at (7.25,-0.5);
\coordinate (L8) at (7.75,-0.5);
\coordinate (L9) at (7.25,-1);
\coordinate (L10) at (8.5,0);
\draw (L1)--(L2)--(L3)--cycle;
\draw (L6)--(L8)--(L7)--(L9)--cycle;
\coordinate (L1) at (9,-1);
\coordinate (L2) at (10.5,-1);
\coordinate (L3) at (10.5,-0.5);
\coordinate (L4) at (9.5,0.5);
\coordinate (L5) at (9,0.5);
\coordinate (L6) at (9,-0.5);
\coordinate (L7) at (9,0);
\coordinate (L8) at (9.5,-1);
\coordinate (L9) at (9.5,-0.5);
\coordinate (L10) at (9.5,0);
\coordinate (L11) at (10,-1);
\coordinate (L12) at (10,-0.5);
\coordinate (L13) at (10,0);
\draw (L1)--(L2)--(L3)--(L4)--(L5)--cycle;
\foreach \t in {1,2,...,13} \fill[black] (L\t) circle (0.03);
\coordinate (A1) at (0,-0.5);
\coordinate (A2) at (0.5,-0.5);
\coordinate (A3) at (0.5,-1);
\coordinate (A4) at (0,-1);
\coordinate (A5) at (1,0);
\coordinate (A6) at (0.5,0);
\coordinate (A7) at (1,-0.5);
\coordinate (B1) at (3,-0.5);
\coordinate (B2) at (3.5,-0.5);
\coordinate (B3) at (3.5,-1);
\coordinate (B4) at (3,-1);
\coordinate (B5) at (4,0);
\coordinate (B6) at (3.5,0);
\coordinate (B7) at (4,-0.5);
\draw[thick, dotted] (A6)--(A2)--(A7);
\draw[thick] (0.5,-0.47)--(1,0.03);
\draw[thick] (0.525,-0.5)--(1.025,0);
\draw[thick] (A1)--(A2)--(A3);
\draw (0,-0.54)--(0.46,-0.54)--(0.46,-1);
\coordinate [label=below:\text{$\Delta_{f_2}$}] (a) at (6,-1.25);
\coordinate [label=below:\text{$\Delta_{g_2}$}] (b) at (7.6,-1.25);
\coordinate [label=below:\text{$\Newt(f_2)+\Newt(g_2)$}] (b) at (9.8,-1.25);
\coordinate [label=below:\text{$V(f_2)$ and $V(g_2)$}] (c) at (0.5,-1.25);
\draw[thick, dotted] (B6)--(B2)--(B7);
\draw[thick] (3.5,-0.47)--(4,0.03);
\draw[thick] (3.525,-0.5)--(4.025,0);
\draw[thick] (B1)--(B2)--(B3);
\coordinate [label=below:\text{$\div(f_2)-\div(g_2)$}] (c) at (3.5,-1.25);
\end{tikzpicture}
\caption{Tropical curves $V(f_2)$ and $V(g_2)$.}
\label{lf2g2}
\end{figure}

\begin{figure}[H]
\centering
\begin{tikzpicture}
\coordinate (L1) at (0.5,2.5);
\coordinate (L2) at (1,2.5);
\coordinate (L3) at (0,3.5);
\coordinate (L4) at (0,3);
\coordinate (L5) at (0.5,3);
\draw (L1)--(L2)--(L3)--(L4)--cycle;
\coordinate (A1) at (1.5,3.1);
\coordinate (A2) at (2.1,3.1);
\coordinate (A3) at (2.5,3.5);
\coordinate (A4) at (2.1,2.5);
\coordinate (A5) at (1.5,2.5);
\draw (A1)--(A2)--(A4);
\draw (A3)--(A5);

\coordinate (L1) at (5,2.5);
\coordinate (L2) at (5.5,2.5);
\coordinate (L3) at (4.5,3.5);
\coordinate (L4) at (4.5,3);
\coordinate (L5) at (5,3);
\draw (L1)--(L2)--(L3)--(L4)--cycle;
\draw (L4)--(L5);
\coordinate (B1) at (6,3.1);
\coordinate (B2) at (6.5,3.1);
\coordinate (B3) at (6.5,2.5);
\coordinate (B4) at (7,3.6);
\coordinate (B5) at (6.2,2.5);
\coordinate (B6) at (7,3.3);
\draw (B1)--(B2)--(B3);
\draw (B2)--(B4);
\draw (B5)--(B6);

\coordinate (L1) at (10,2.5);
\coordinate (L2) at (10.5,2.5);
\coordinate (L3) at (9.5,3.5);
\coordinate (L4) at (9.5,3);
\coordinate (L5) at (10,3);
\draw (L1)--(L2)--(L3)--(L4)--cycle;
\draw (L1)--(L5);
\coordinate (A1) at (11,3);
\coordinate (A2) at (11.75,3);
\coordinate (A3) at (11.75,2.5);
\coordinate (A4) at (12,3.25);
\coordinate (A5) at (11,2.6);
\coordinate (A6) at (11.9,3.5);
\draw (A1)--(A2)--(A3);
\draw (A4)--(A2);
\draw (A5)--(A6);

\coordinate (L1) at (0.5,1);
\coordinate (L2) at (1,1);
\coordinate (L3) at (0,2);
\coordinate (L4) at (0,1.5);
\coordinate (L5) at (0.5,1.5);
\draw (L1)--(L2)--(L3)--(L4)--cycle;
\draw (L2)--(L4);
\coordinate (B1) at (1.6,1.05);
\coordinate (B2) at (1.9,1.35);
\coordinate (B3) at (2.06,1.67);
\coordinate (B4) at (2.36,1.97);
\coordinate (B5) at (1.9,1);
\coordinate (B6) at (1.6,1.67);
\draw (B1)--(B2)--(B3)--(B4);
\draw (B2)--(B5);
\draw (B3)--(B6);

\coordinate (L1) at (5,1);
\coordinate (L2) at (5.5,1);
\coordinate (L3) at (4.5,2);
\coordinate (L4) at (4.5,1.5);
\coordinate (L5) at (5,1.5);
\draw (L1)--(L2)--(L3)--(L4)--cycle;
\draw (L3)--(L1);
\coordinate (B1) at (6,1);
\coordinate (B2) at (6.4,1.4);
\coordinate (B3) at (6.7,1.55);
\coordinate (B4) at (7,1.85);
\coordinate (B5) at (6.7,1);
\coordinate (B6) at (6,1.4);
\draw (B1)--(B2)--(B3)--(B4);
\draw (B3)--(B5);
\draw (B2)--(B6);

\coordinate (L1) at (10,1);
\coordinate (L2) at (10.5,1);
\coordinate (L3) at (9.5,2);
\coordinate (L4) at (9.5,1.5);
\coordinate (L5) at (10,1.5);
\draw (L1)--(L2)--(L3)--(L4)--cycle;
\draw (L4)--(L5);
\draw (L1)--(L5);
\coordinate (A1) at (11,1);
\coordinate (A2) at (11.4,1.4);
\coordinate (A3) at (11.7,1.4);
\coordinate (A4) at (12,1.7);
\coordinate (A5) at (11.7,1);
\coordinate (A6) at (11,1.7);
\coordinate (A7) at (11.4,1.7);
\coordinate (A8) at (11.7,2);
\draw (A1)--(A2)--(A3)--(A4);
\draw (A3)--(A5);
\draw (A2)--(A7);
\draw (A6)--(A7)--(A8);

\coordinate (L1) at (0.5,-0.5);
\coordinate (L2) at (1,-0.5);
\coordinate (L3) at (0,0.5);
\coordinate (L4) at (0,0);
\coordinate (L5) at (0.5,0);
\draw (L1)--(L2)--(L3)--(L4)--cycle;
\draw (L4)--(L5);
\draw (L4)--(L2);
\coordinate (B1) at (1.5,-0.6);
\coordinate (B2) at (1.8,-0.3);
\coordinate (B3) at (1.95,0);
\coordinate (B4) at (2.25,0.3);
\coordinate (B5) at (1.8,-0.6);
\coordinate (B6) at (1.5,0.3);
\coordinate (B7) at (1.95,0.3);
\coordinate (B8) at (2.25,0.6);
\draw (B1)--(B2)--(B3)--(B4);
\draw (B2)--(B5);
\draw (B3)--(B7);
\draw (B6)--(B7)--(B8);

\coordinate (L1) at (5,-0.5);
\coordinate (L2) at (5.5,-0.5);
\coordinate (L3) at (4.5,0.5);
\coordinate (L4) at (4.5,0);
\coordinate (L5) at (5,0);
\draw (L1)--(L2)--(L3)--(L4)--cycle;
\draw (L1)--(L5);
\draw (L1)--(L3);
\coordinate (B1) at (5.9,-0.35);
\coordinate (B2) at (6.3,0.05);
\coordinate (B3) at (6.6,0.2);
\coordinate (B4) at (6.9,0.2);
\coordinate (B5) at (7.2,0.5);
\coordinate (B6) at (6.9,-0.35);
\coordinate (B7) at (5.9,0.05);
\coordinate (B8) at (6.9,0.5);
\draw (B1)--(B2)--(B3)--(B4)--(B5);
\draw (B4)--(B6);
\draw (B3)--(B8);
\draw (B2)--(B7);
\end{tikzpicture}
\caption{The eight subdivision of $\Delta$ and the shapes of their dual tropical curves.}
\label{f9}
\end{figure}

Thus, a pair of tropical polynomials such that it is the minimum volume expression of a tropical rational function does not necessarily have the smallest factorization complexity.

\end{document}